\documentclass{amsart}

\usepackage{amssymb}
\usepackage[stretch=10,shrink=10]{microtype}
\usepackage[showframe=false, margin = 1.5 in]{geometry}
\usepackage{mathtools}
\usepackage[shortlabels]{enumitem}
\usepackage{tikz}
\usetikzlibrary{arrows.meta}
\usepackage{mathrsfs}
\usepackage{hyperref}
\usepackage{theoremref}
\usepackage{comment}
\usepackage{mathdots}

\usetikzlibrary{decorations.pathmorphing}
\usepackage[detect-all]{siunitx}

\tikzset{
   ragged border/.style={ decoration={random steps, segment length=1mm, amplitude=0.5mm},
           decorate,
   }
}

\let\oldtocsubsection=\tocsubsection
\renewcommand{\tocsubsection}[2]{\hspace*{1.1cm}\oldtocsubsection{#1}{#2}}

\makeatletter
\def\thmref@flush{%
   \ifx\thmref@last\empty\else
      \ifthmref@comma, \thmref@finaltrue\fi \thmref@commatrue
      \thmref@last \ifx\thmref@stack\empty\else s\fi \thmref@num 0
      \let\do\thmref@one \thmref@stack
      \ifcase\thmref@num\or\space and\else\thmref@finaltrue, and\fi
      ~\ref{\thmref@head}\let\thmref@stack\empty\fi}
\def\thmref@one#1{\ifnum\thmref@num>0,\fi
   \space\ref{#1}\advance\thmref@num 1\relax}
\makeatother

\newcommand{\rootvertex}{\varnothing}
\renewcommand{\emptyset}{\varnothing}
\newcommand{\E}{\mathbf{E}}
\renewcommand{\P}{\mathbf{P}}
\newcommand{\TT}{\mathbb{T}}
\newcommand{\ZZ}{\mathbb{Z}}

\newcommand{\Ss}{\mathcal{S}}

\newcommand{\Vv}{\mathcal{V}}
\newcommand{\Ll}{\mathcal{L}}

\newcommand{\Fff}{\mathscr{F}}
\newcommand{\Cc}{\mathcal{C}}
\newcommand{\1}{\mathbf{1}}
\newcommand{\f}{\frac}

\newcommand{\NN}{\mathbb{N}}

\newcommand{\FM}{\mathrm{FM}}

\newcommand{\oI}{\overline{I}}
\DeclareMathOperator{\Poi}{Poi}
\DeclareMathOperator{\Ber}{Ber}
\DeclareMathOperator{\Bin}{Bin}
\DeclareMathOperator{\Geo}{Geo}
\DeclarePairedDelimiter\abs{\lvert}{\rvert}%
\DeclarePairedDelimiter\floor{\lfloor}{\rfloor}%
\DeclarePairedDelimiter\ceil{\lceil}{\rceil}%
\newcommand{\pgfprec}{\preceq_{\text{pgf}}}
\newcommand{\pgfsucc}{\succeq_{\text{pgf}}}

\newtheorem{thm}{Theorem}[section]
\newtheorem{lemma}[thm]{Lemma}
\newtheorem{prop}[thm]{Proposition}
\newtheorem{cor}[thm]{Corollary}

\newtheorem{question}[thm]{Question}

\theoremstyle{remark}

\theoremstyle{definition}
\newtheorem{define}[thm]{Definition}

\newcommand{\prob}{\P}
\newcommand{\cover}{\mathcal C}

\newcommand{\bigmid}{\;\big\vert\;}

\newcommand{\biggmid}{\:\bigg\vert\:}

\newcommand{\eqd}{\overset{d}=}
\makeatletter
\newcommand*\proc{{\mathpalette\bigcdot@{.65}}}
\newcommand*\bigcdot@[2]{\mathbin{\vcenter{\hbox{\scalebox{#2}{$\m@th#1\bullet$}}}}}
\makeatother

\begin{document}

\title{Cover time for the frog model on trees}

\author {Christopher Hoffman}
\address{Department of Mathematics, University of Washington}
\email{hoffman@math.washington.edu}

\author{Tobias Johnson}
\address{Department of Mathematics, College of Staten Island}
\email{tobias.johnson@csi.cuny.edu}

\author{Matthew Junge}
\address{Department of Mathematics, Duke University}
\email{jungem@math.duke.edu}
\thanks{C.H.\ received support from NSF grant DMS-1308645 and from a Simons Fellowship.
  T.J.\ received support from NSF grants DMS-1401479
  and DMS-1811952 and PSC-CUNY Award \#61540-00~49.}




\begin{abstract}
  The frog model is a branching random walk on a graph in which particles branch only
  at unvisited sites. Consider an initial particle density of $\mu$ on the full $d$-ary tree of height $n$. 
If $\mu= \Omega( d^2)$, all of the vertices are visited in time $\Theta(n\log n)$ with high probability. 
  Conversely, if $\mu = O(d)$ the cover time is 
  $\exp(\Theta(\sqrt n))$ with high probability. 
\end{abstract}

\maketitle


\section{Introduction}

The \emph{frog model} is 
a system of interacting walks that starts with one
particle awake at the root of a graph and some number, typically Poisson-distributed with mean $\mu$, of sleeping particles at all the other vertices. Wakened particles perform simple random walk in discrete time. They wake any sleeping particles they encounter, which then begin their own independent random walks. A long-open problem posed to us several years ago by Itai Benjamini has been to determine the time it takes to visit every vertex of the full $d$-ary tree of height~$n$ (i.e., the cover time). 
One might expect a simple argument would establish fast or slow cover times when the density of particles is very high or small.
In fact, for \emph{every} density of particles, it was unknown even if this quantity
was polynomial or superpolynomial in $n$. Here we demonstrate that both can occur.

If we view the process as modeling the spread of an infection, finite graphs are its natural setting and cover time a fundamental measurement. Finite trees are particularly interesting because of the phase transition that occurs on infinite rooted $d$-ary trees with an average of $\mu$ particles per site:
As we increase $\mu$, the root goes from being visited finitely to infinitely many times 
\cite{HJJ1, HJJ2}. Moreover, the companion to this work \cite{HJJ_shape} proves that when $\mu = \Omega(d^2)$, the root is visited at a linear rate.  The dramatically different regimes on infinite trees suggest that both fast and slow cover times should occur on finite trees \cite{Fot?, JJ3_log}.  However, it is unclear how reflection at the leaves influences the spread of active particles. Indeed, dealing with the boundary is the biggest obstacle to establishing regimes for fast and slow cover times.


%

First we describe what was previously known. The cover time is trivially at least $n$, and it is bounded above 
by the cover time for a single random walk on a tree, which is exponential in $n$ \cite{aldous}.
Until recently, these were the only known results. 
For any fixed $d$ and particle density, Hermon improved the lower bound to $\Omega(n \log n)$ and 
the upper bound to
$\exp(O(\sqrt {n}))$ \cite{Fot?}. 
In this paper, we prove that if the density of particles is sufficiently large then Hermon's lower bound is sharp, and if the density is small then his upper bound is sharp. In particular, this is the first proof that there exists a $d\geq 2$ and density of particles for which the cover time is polynomial, or a $d\geq 2$ and density of particles for which the cover time is superpolynomial.

We mention a few other closely related topics.
The \emph{susceptibility} of the frog model on a finite graph
is the minimum lifespan of frogs such that all sites are visited. This statistic has been
studied on tori and expanders \cite{benjamini2016epidemic} and on trees \cite{Fot?}.
In none of these cases does it exhibit a phase transition in the density of particles, making it
qualitatively very different from cover time.

A process resembling the frog model was proposed by Benjamini to study the connectivity of social
networks and the spread of epidemics and has been studied on finite graphs \cite{benjamini2016rapid}
and infinite graphs \cite{infinite_social}. On infinite nonamenable graphs, there is a phase transition 
in the initial density for whether all particles are eventually socially connected. For vertex-transitive
amenable graphs, there is not. This resembles the frog model, which has a phase transition between
transience and recurrence on trees \cite{HJJ2}, but not on lattices \cite{random_shape}.

\subsection*{Result}

As we mentioned, \cite{Fot?} gives
the first nontrivial upper and lower bounds on the cover time, which we now
state in more detail. Let $\Poi(\mu)$ denote  a Poisson distribution with mean $\mu$. We refer to the frog model with one frog awake at the root, i.i.d.-$\Poi(\mu)$ frogs elsewhere, and frogs following independent simple random walk paths as having \emph{i.i.d.-$\Poi(\mu)$
initial conditions}. We let $\TT_d^n$ denote the rooted, full $d$-ary tree of height~$n$.
This is the tree with levels $0,\ldots, n$ in which all vertices in levels $0,\ldots,n-1$
have $d$ children.

Let $\cover = \cover(n,d,\mu)$ denote the cover time for the frog model on $\TT_d^n$ with i.i.d.-$\Poi(\mu)$
initial conditions. In \cite[Theorem 2]{Fot?}, Hermon proves there exists a constant $c>0$ such that
for any $\mu>0$ and $d\geq 2$,
\begin{align*}
  \lim_{n\to\infty} \P\Bigl[\cover\leq e^{c\sqrt{n\log d}}\Bigr] = 1.
\end{align*}
As for the lower bound, it follows from \cite[Theorem~1]{Fot?} that there exists a constant
$C>0$ such that for any $\mu>0$ and $d\geq 2$,
\begin{align*}
  \lim_{n\to\infty}\P\biggl[\cover \geq \frac{Cn\log n}{\mu} \biggr] = 1.
\end{align*}

 We now give our main result, which demonstrates
the existence of two distinct behaviors for the cover time
depending on the initial density of frogs. With a high density, the cover time is
$O_d(n\log n/\mu)$ with high probability. By Hermon's lower bound, this determines the cover time
up to constant factor for each fixed choice of $d$. 
With a low initial density of frogs, we prove that the
cover time is $\exp\bigl(\Omega(\sqrt{n\log d})\bigr)$ with high probability, 
which is sharp up to the constant in the exponent by Hermon's upper bound. 
In fact,  \cite[Theorem 2]{Fot?} also gives an upper bound for the cover time when $\mu$ decays in the height of the tree; one can take $\mu$ as small as $\exp(- \sqrt{n \log d})$ and still obtain a bound of the same order. 
Thus, our lower bound shows that for small but fixed values of $\mu$, the cover time exhibits the
same asymptotic behavior as when $\mu$ decays rapidly as $n$ grows.

\begin{thm}\thlabel{thm:main}
  Let $\cover=\cover(n,d,\mu)$ denote the cover time for the frog model on $\TT_d^n$ with one awake frog at the root and i.i.d.-$\Poi(\mu$) conditions.
  \begin{enumerate}[(a)]
    \item There exist constants $\beta_0$, $C_d$, and $n_0(\mu,d)$ \label{i:main.upper}
      such that for all $d\geq 2$ and $\mu> \beta_0d^2$,
      \begin{align*}
        \prob\biggl[\cover \leq \frac{C_d n \log n}{\mu}\biggr] \geq 1- d^{-n}
      \end{align*}
      for $n\geq n_0(\mu,d)$.
    \item \label{i:main.lower}
  Suppose that $\mu\leq\min(d^{1-\epsilon},d/100)$ for $\epsilon\in(0,1]$.
  For some absolute constant $c>0$, 
  \begin{align*}
    \P\bigl[ \cover \geq e^{c\sqrt{\epsilon n\log d}} \bigr] &\geq 1 - e^{-c\sqrt{\epsilon n\log d}}
  \end{align*}
  for $n\geq \log d/c^2\epsilon$.
  \end{enumerate}
\end{thm}

Our full versions of these bounds, \thref{thm:upper,thm:lb}, are slightly
stronger in that we extend them to initial distributions other than Poisson.
We note that our lower bound, part~\ref{i:main.lower}, shows that the cover
time is large even when $\mu$ is as large as $d/100$; 
see \thref{cor:lb.linear}.

Thus, our results establish a slow cover time regime when $\mu=O(d)$ and
a fast cover time regime when $\mu=\Omega(d^2)$. This raises the question
of what happens in between. On the infinite tree, the threshold
between recurrence and transience occurs when $\mu$ is on the order of $d$
\cite{JJ3_log}. 
This paper's results are consistent with the possibility that the slow
and fast cover time regimes on the finite tree occur at the same
parameters as the transient and recurrent phases on the infinite tree.
But it is not clear this is so.
\begin{question}
  Are there other phases for the cover time of the frog model on $\TT^n_d$ besides those described
  in this paper? Is there a sharp phase transition
  between phases? If so, how does the process behave at critical values of $\mu$?
\end{question}
In \cite{Dickman}, the \emph{activated random walk process},
which is essentially the frog model where particles fall back asleep at random,
is discussed in connection with \emph{self-organized criticality}, a phenomenon in which
some real physical systems naturally push themselves toward criticality.
The idea is that while conservative systems (in which particles are neither created nor destroyed)
do not exhibit self-organized criticality, their behavior at criticality can nonetheless
be a good model for it (see also \cite[Section~1.3]{RSZ}, whose discussion is aimed at mathematicians).
This makes the frog model's behavior
at criticality on both finite and infinite trees a 
particularly intriguing topic.

\subsection*{Description of proof}


In Section~\ref{sec:upper}, we tackle the cover time upper bound. 
The starting point for the proof is that the infected region (i.e., set of visited sites) grows linearly
for the frog model on an infinite tree, which we prove in the companion paper \cite{HJJ_shape}.
Naively, one might think that a polynomial cover time bound would follow as an easy corollary,
but we do not believe there is a quick argument. 
The issue is that our strong recurrence results from \cite{HJJ_shape}, that the number of visits
to the root grows linearly in time, become less powerful as they are applied to a finite tree
near its leaves. We describe our argument in detail here to illustrate the problem at the boundary and our resolution of it. In the rough description below, we will suppress the fact that the constants depend on $d$ and write $O(\cdot)$ rather than $O_d(\cdot)$.

 Looking towards a union bound, we must show it exponentially likely that an arbitrary leaf 
 $v_0 \in \mathbb T_d^n$ is woken in time $O(n\log n)$.  Consider the spine   $v_0,v_1,\hdots,v_n = \emptyset$ leading from $v_0$ to the root~$\varnothing$. To $\TT_d^n(v_k)$, the subtree rooted at $v_k$, we
 attach a random variable $I_k$ defined as the number of frogs that must enter $\TT_d^n(v_k)$ 
 to accumulate frogs that are frozen at $v_{k-1}$ at a linear rate for $d^k$ time steps. 
 A possibly helpful metaphor is cascading water down a stair-step fountain (see Figure \ref{fig:cascade}). Each basin needs a certain amount of water to reach a tipping point, after which it will pour water steadily into the one below it.
 In our proof, we wait until $I_{n-1}$ frogs have accumulated at $v_{n-1}$. By definition of 
 $I_{n-1}$, this sets off a linear flow of frogs which will  send $I_{n-2}$ frogs to $v_{n-2}$ in $O(I_{n-2})$ steps. This cascade continues until $cn \log n$ frogs have accumulated at site $v_J$, at a distance 
 $J \approx  \log_dn +\log_d\log n +C_d $ from $v_0$. 
At this point, we have built up enough frogs that we can ignore the wake-up dynamics of the frog model
and instead show it is exponentially
likely that at least one of $cn\log n$ random walks started at $v_J$ will visit $v_0$ in the next
$O(n\log n)$ time steps. We can then apply a union bound over all leaves of $\TT_d^n$.

 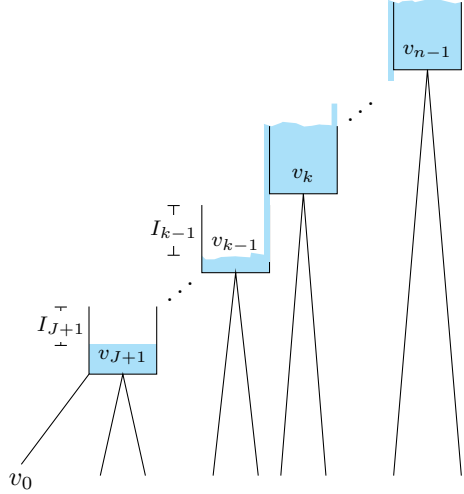
\begin{figure}
 	\begin{tikzpicture}[scale = .15]
 	
\begin{scope}[shift={(11 ,11)}]
\fill[cyan!30] (0,6) -- (-.5,6) -- (-.5,-1)-- (0,-1)
 -- cycle;
  \fill[cyan!30]
        decorate[ragged border]{
        (-.5,6.2) -- (6,6)
        }
        -- (6,1) -- (6,0.5) --(6,0) -- (0,0) -- cycle;

  \draw (0,6) -- (0,0) -- (6,0) -- (6,6) ;
  
   \node[anchor=north,font=\footnotesize] at (3,3) {$v_{n-1}$};
  \node at (-8,-8) {$\iddots$};
	\draw (0,-36) -- (3,0) -- (6,-36);
\end{scope}

\begin{scope}[shift={(-6 ,-7)}]
  \fill[cyan!30]
        decorate[ragged border]{
        (0,1.5) -- (6,1.5)
        }
        -- (6,1) -- (6,0.5) --(6,0) -- (0,0) -- cycle;
  \draw (0,6) -- (0,0) -- (6,0) -- (6,6) ;
  \draw[|-|] (-2.5,1.5) --
        node[fill=white,font=\footnotesize,inner ysep=2pt,inner
                xsep=0]{$I_{k-1}$}(-2.5,6);
  \node[anchor=north,font=\footnotesize] at (3,4) {$v_{k-1}$};
	\draw (1,-18) -- (3,0) -- (5,-18);
\end{scope}
\begin{scope}[shift={(0 ,0)}]
\fill[cyan!30] (0,6) -- (-.5,6) -- (-.5,-6)-- (0,-6)
 -- cycle;
\fill[cyan!30] (6,8) -- (5.5,8) -- (5.5,5)-- (6,5.5)
 -- cycle;
  \fill[cyan!30]
        decorate[ragged border]{
        (-.5,6.2) -- (6,6)
        }
        -- (6,1) -- (6,0.5) --(6,0) -- (0,0) -- cycle;

  \draw (0,6) -- (0,0) -- (6,0) -- (6,6) ;
  \node[anchor=north,font=\footnotesize] at (3,3) {$v_k$};
  \node at (-8,-8) {$\iddots$};
  \node at (8,8) {$\iddots$};
  \draw (1, -25) -- (3,0) -- (5,-25);
\end{scope}

\begin{scope}[shift={(-16 ,-16)}]
  \fill[cyan!30]
        decorate[]{
        (0,2.6) -- (6,2.6)
        }
        -- (6,1) -- (6,0.5) --(6,0) -- (0,0) -- cycle;
  \draw (0,6) -- (0,0) -- (6,0) -- (6,6) ;
  \draw[|-|] (-2.5,2.5) --
        node[fill=white,font=\footnotesize,inner ysep=2pt,inner
                xsep=0]{$I_{J+1}$}(-2.5,6);
  \node[anchor=north,font=\footnotesize] at (3,3) {$v_{J+1}$};
  \draw (1,-9) -- (3,0) -- (5,-9);
  \draw (0,0) -- (-6,-8);
  \node[below] at (-6,-8) {$v_0$};
\end{scope}

\end{tikzpicture}%
\caption{We show that the cover time is $O(n\log n)$ by creating a cascade of frogs that move towards an arbitrary leaf $v_0$ along the path $v_0,\hdots,v_n= \emptyset$.
Once $I_k$ frogs build up at $v_k$, a constant stream of frogs flows to $v_{k-1}$.
Using strong recurrence, we show that $I_{n-1}$ frogs build up at $v_{n-1}$ in time $O(n)$. This sets the cascade in motion, initiating a constant flow of frogs to $v_{n-2}$.
After $O(I_{n-2})$ steps, we have built up $I_{n-2}$ frogs at $v_{n-2}$, 
setting off the next stage of the cascade,
and so on. This quickly builds up $\Omega_d(n\log n)$ frogs at distance $J= O(\log n)$ from $v_0$,
and it is exponentially likely that at least one will visit $v_0$ in the next $O(n\log n)$ steps.} \label{fig:cascade}
 \end{figure}

The argument outlined above shows that the cover time is roughly $O(I_{n-1}+\cdots+I_{J+1}+ n \log n)$.
Deducing a fast cover time is thus reduced to bounding the random variables $(I_k)$.
Since $I_k$ is determined by the frogs within $\TT_d^n(v_k)$, the random variables
$(I_k)$ are independent.
We show that $I_k$ has an exponential tail independent of $k$,
which implies $I_{n-1} + \cdots + I_{J+1} = O(n)$ with exponentially high probability by a Chernoff bound.
The proof that $I_k$ has an exponential tail uses strong recurrence  but is not an easy corollary of it. 
The issue is that strong recurrence only guarantees a steady flow of frogs out of $\TT_d(v_k)$ up to time $k$,
while we need a flow up to time $d^k$, or else the argument would not work for $v_k$ close to the leaves.
Indeed, for $k\approx J$, strong recurrence yields a flow out of $\TT_d(v_k)$ for only
$O(\log n)$ steps, rather than the $d^k\approx n\log n$ steps that we need.
Thus, our challenge is to show that a steady flow of frogs
out of $\TT_d^n(v_k)$ persists for much longer than given by strong recurrence.
This argument makes up the bulk of Section~\ref{sec:upper}.

In Section~\ref{sec:lower}, we give our bound for the low density case.
Our argument has no precursors in published work, as far as we know.
We consider the number of visits to the root for the frog model on $\TT^{j^2}_d$ in the first
$2^j$ steps. We inductively assume that the expected
number of visits to the root is $O(1)$, and we then try to prove that 
this estimate continues to hold for the frog model
on $\TT^{(j+1)^2}_d$ after $2^{j+1}$ time steps. To do this, we separate the tree into its
first $O(j)$ and its final $j^2$ levels. We then push the induction forward by
bounding the growth of frogs at different times in the two parts of the tree by various combinations of
the inductive hypothesis, a bound given by 
branching random walk, and a bound of assuming all frogs are awake in a given subtree. \thref{thm:main}\ref{i:main.lower} follows from considering $n \approx j^2$. 

As we mentioned earlier, we also obtain results when the sleeping frog distributions
at each vertex are not Poisson. These results are easy applications of \cite{JJ3_order},
in which we show that increasing the initial distributions in various stochastic orders
causes certain statistics of the frog model to increase as well. We give a further introduction
to these techniques in Appendix~\ref{sec:comparison}. Some facts for random walk decompositions on trees and concentration inequalities are also contained in the appendices.

\section{Preliminaries}
Here we describe our notation, certain variants of the frog model, and also results that we will need from \cite{HJJ_shape}.

\subsection*{Notation}


For our purposes the frog model takes place on either the infinite rooted $d$-ary tree $\TT_d$ or on the full $d$-ary tree of height~$n$ denoted by $\TT_d^n$. The root of whatever tree we are discussing will be denoted by $\varnothing$. Vertices at distance~$k$ from the root
are at \emph{level~$k$} of the tree. For any rooted tree~$T$ and vertex $v\in T$, 
we denote the subtree of $T$ made up of $v$ and its descendants by $T(v)$.

Formally, the frog model is a pair $(\eta,S)$ where 
for each vertex $v$ other than the starting one, $\eta(v)$ is the number of frogs initially 
sleeping at $v$, and $S=(S_\proc(v,i))_{v\in G,i\geq 1}$ is a collection of walks
satisfying $S_0(v,i)=v$.
The $i$th particle sleeping at $v$ on waking follows the path $S_\proc(v,i)$.
 When we discuss the frog model on a given graph with, say, i.i.d.-$\Poi(\mu)$ initial
conditions, unless we say otherwise we assume that the paths are simple random walks, and all of the random variables are independent.
The root is assumed
to be the starting vertex unless stated otherwise.
The frog model evolves in discrete time,
though it is easy to show that the results of this paper hold in continuous time as well.
A realization of the frog model is called either \emph{transient} or \emph{recurrent} depending 
on whether the starting vertex is visited infinitely often by frogs.
The \emph{cover time} of a given frog model is the random variable defined as the first
time all vertices in the system have been visited. Traditionally, particles are referred to as frogs, a practice we continue.

We let $\Geo(p)$ be the distribution that places probability $(1-p)^k p$ on $k \geq 0$.
We also refer to the geometric distribution on $\{1,2,\ldots\}$ with parameter $p$, 
which is the same distribution shifted by one.
In a mild abuse of notation, we sometimes use $\Poi(\mu)$ and $\Bin(n,p)$ to refer to random
variables with the given distributions rather than the distributions themselves, as in statements
like  $\P[\Poi(\mu) = 0]=e^{- \mu}$.

\subsection{Modified frog models} \label{sec:modified.frog.models}

At times in our argument, it is helpful to consider variants of the frog model that couple to the original process.
%
A \emph{stopped version} of a given frog model $(\eta,S)$ is a frog model
$(\eta,S')$ where each path $S'_\proc(v,i)$ consists of $S_\proc(v,i)$ stopped at some time $T(v,i)\in\NN\cup\{\infty\}$.
These must be stopping times for the frog model, in the sense that the decision to stop a frog at some
time must be determined from the history of the stopped process up to that time. We give a quick sketch of
how to formalize this. 
Following \cite[Section~2]{KZ}, let $W'_j(\eta,S)$ be the set of sites visited for the first time
at time~$j$ in the stopped process. Define $\Fff_t$ as the $\sigma$-algebra representing all information about
the stopped process revealed by time~$t$; formally, it is generated by the sets $W'_j(\eta,S)$ for
$j\in\{0,\ldots,t\}$, the frog counts $\eta(v)$ for $v\in\cup_{j=0}^t W'_j(\eta,S)$, and the
frog paths $(S'_k(v,i))_{k=0}^{t-j}$ for each $j\in\{0,\ldots,t\}$ and $v\in W'_j(\eta,S)$.
We require the event $\{T(v,i)\leq t\}$ to be measurable with respect to $\Fff_t$.
As a consequence of this definition, for a stopped version of a frog model with,
say, simple random walk paths, we can unstop all frogs at a given time and have them
continue as independent simple random walks, since the stopping times do not impart
any conditioning on the future part of the paths. By an easy coupling, in any given time stochastically fewer
frogs are woken in the stopped process than in the original one.

When proving lower bounds on the growth of the frog model on the infinite tree $\TT_d$,
we typically work with what we call the self-similar frog model. Roughly speaking, it is
the frog model with nonbacktracking frog paths, where frogs are stopped so that
for each subtree of the form $\TT_d(v)$,
at most one frog from outside the subtree is allowed to enter it.
To define it rigorously, we first define
a \emph{uniform nonbacktracking random walk} as a nearest neighbor path that samples
uniformly from all adjacent edges on its first step, and then thereafter
samples uniformly from all adjacent
edges except the one just traversed. 
On $\TT_d$, this is particularly simple: the path moves towards the root for some random amount of time, then takes a random nonbacktracking step away from the root, and then
follows a uniformly sampled geodesic to $\infty$.

To define the \emph{self-similar frog model on $\TT_d$}, first let the frog paths be independent
uniform nonbacktracking random walks. Now, we stop frogs as follows to enforce
the rule that at most one frog enters any subtree.
On a given step of the frog model,
suppose that some vertex~$v\in\TT_d\setminus\{\rootvertex\}$ is visited for the first time.
Let $v'$ be the parent of $v$. On this step, one or more frogs move from $v'$ to $v$.
Stop all but one of them, and on all subsequent steps stop all frogs on moving from $v'$ to $v$.
Additionally, stop all frogs at $\rootvertex$ at steps 1 and beyond.
We refer to \cite[Section~2.1]{HJJ_shape} and \cite[Section~3.1.1]{JJ3_log}
for more background information about the self-similar frog model.
The reason for calling it self-similar is that only one external frog, i.e., a frog initially at a vertex in $\mathbb T_d \setminus \mathbb T_d(v)$,
may enter each $\TT_d(v)$. Because the frog paths are non-backtracking, the process on $\{v'\}\cup\TT_d(v)$ from the time
a frog moves from $v'$ to $v$ is identical in law as
on $\{\rootvertex\}\cup\TT_d(\rootvertex')$ from step~1 onward. Here $\rootvertex'$
is the child of $\rootvertex$ visited by the initial frog on its first step.

The self-similar frog model is defined on the infinite tree $\TT_d$, though
we will sometimes consider it on the finite tree $\TT_d^n$ by 
freezing frogs at leaves. But in proving our upper bound on cover time,
we will usually consider a different process we call
the nonbacktracking frog model on $\TT_d^n$. To describe it, we first define a \emph{root-biased nonbacktracking random walk from $v_0$ on
$\TT_d^n$} as a walk distributed as follows. We set $X_0=v_0$, and then we choose
$X_1$ uniformly from the neighbors of $X_0$. Conditionally on $X_0,\ldots,X_i$, we choose
$X_{i+1}$ as follows: If $X_i=\varnothing$, choose $X_{i+1}$ to be $X_{i-1}$ with probability
$1/d^2$ and to be each of the other children of the root with probability
$(d+1)/d^2$. If $X_i$ is a leaf, then set $X_{i+1}$ to be its parent (the only possibility
for the next step).
Otherwise, choose $X_{i+1}$ uniformly from the neighbors of $X_i$ other than $X_{i-1}$.
It turns out that a simple random walk decomposes into this path plus independent
excursions off of it (see \cite[Appendix~A]{HJJ_shape}).
The odd behavior at the root results from the asymmetry of the tree there.

Finally, define the \emph{nonbacktracking frog model} on $\TT^n_d$ as the frog model 
whose paths are independent root-biased nonbacktracking random walks on the specified tree.
The following result shows
that the time change for the underlying random walks speeds up the nonbacktracking model
by only a constant factor compared to the usual frog model.
This allows us to work with nonbacktracking frog models
when we prove \thref{thm:main}\ref{i:main.upper}.
\begin{prop}[{\cite[Proposition~2.2]{HJJ_shape}}]\thlabel{cor:time.change}
  Let $(\eta,S)$ and $(\eta,S')$ be respectively the usual and the nonbacktracking frog models
  on $\TT^n_d$, with arbitrary initial configuration $\eta$.
  There exists a coupling of the frog models $(\eta,S)$ and $(\eta,S')$ such that the following holds:
  For any $b>\log d$, there exists $C=C(b)$ such that all vertices visited in $(\eta,S')$ by time~$t$
  are visited in $(\eta,S)$ by time~$Ct$ with probability $1-e^{-bt}$.
\end{prop}

\subsection{Adaptations of results from \cite{HJJ_shape}} \label{sec:exterior}
We start by stating the result \cite[Theorem 3.1]{HJJ_shape} which demonstrates a strong recurrence regime on infinite $d$-ary trees. We define the \emph{return process} to be a point process on $\mathbb R$ in which each point at $t$ represents a frog that is occupying the root at time $t$. Note that this is supported on the nonnegative integers. 

\begin{thm}\thlabel{thm:strong}
  Consider the self-similar frog model on $\TT_d$ with 
  i.i.d.-$\Poi(\mu)$ initial conditions.
  For any $d\geq 2$, $\alpha>0$, and $\mu\geq 3d(d+1)+\alpha(d+1)$, the return process
  stochastically dominates a Poisson point process with intensity
  measure $\sum_{k=1}^{\infty}\alpha\delta_{2k}$.
\end{thm}
This extends to $\TT^n_d$, but because of the boundary, only up to time $2n-2$.

\begin{cor}\thlabel{cor:strong.finite}
  Consider the self-similar frog model on $\TT^n_d$
  with i.i.d.-$\Poi(\mu)$ initial conditions and frogs frozen on reaching a leaf.
  For any $d\geq 2$, $\alpha>0$, and $\mu\geq 3d(d+1)+\alpha(d+1)$, the return process
  stochastically dominates a Poisson point process with intensity
  measure $\sum_{k=1}^{n-1}\alpha\delta_{2k}$.
\end{cor}
\begin{proof}
  Couple the processes of \thref{thm:strong} and of this corollary by having all frogs
  follow the same paths until reaching the boundary of the finite tree.
  Consider a root visit on the process on the infinite tree occurring before time~$2n$.
  The combined path of frogs waking the returner together with the returning path does
  not reach depth $n$ of the tree, since the return occurs before time $2n$. 
  Thus, the return occurs in the process on the finite tree
  as well. The finite tree process therefore has all of the returns of the infinite tree process
  before time~$2n$, and the result follows from \thref{thm:strong}.
\end{proof}

Last, we state \cite[Lemma 4.1]{HJJ_shape}, which helps us deduce a weaker version of the shape theorem \cite[Theorem 1.1]{HJJ_shape} for the finite tree. 

\begin{lemma}\thlabel{lem:speed}
  Let $\beta >0$ and consider the self-similar frog model on $\TT_d$ with i.i.d.-$\Poi(\mu)$ frogs per site,
  where $\mu=(3+\beta)d(d+1)$.
  Let $\varnothing,v_0,v_1,v_2,\ldots$ be an arbitrary ray in $\TT_d$,
  and condition the initial frog to take its first step to $v_0$.
  Let $\tau_{i}$ be the number of steps after $v_{i-1}$ is first visited that
  $v_i$ is first visited.
  Then $(\tau_i)_{i\geq 1}$ are i.i.d.\ and satisfy
  \begin{align}\label{eq:speed}
    \P[\tau_i > 2t-1] &\leq e^{-\beta t}
  \end{align}
  for $t\in\{1,2,\ldots\}$.
\end{lemma}
An important corollary for our work here is that a self-similar frog model activates half of the leaves in the active branch of any height tree in time $O(k)$ with probability at least $1/2$
(recall that only sites in $\TT_d(\rootvertex')$ are visited in the self-similar model, 
where $\rootvertex'$ is the child of the root first visited by the initial frog). 
\begin{cor}
\thlabel{lem:pos_frac}
Consider a self-similar frog model on $\TT_d^{k+1}$ with i.i.d.-$\Poi(\mu)$ initial
conditions where frogs are frozen at leaves, for any $k\geq 1$.
For $\mu\geq(3+\beta)d(d+1)$ and sufficiently large absolute constants $\beta$ and $C$, 
there exists $p=p(\beta, C)$ such that $d^{k}/2$ of the leaves 
are visited in $C k$ steps with probability at least $p$. Moreover,
$p$ can be made arbitrarily close to $1$ by choosing $C$ and $\beta$ sufficiently large,
and in particular $p\geq 1/2$ when $\beta\geq 2$ and $C\geq 8$.
\end{cor}
\begin{proof}
  Let $v_{-1}=\varnothing$.
  Let $v_0$ be a child of $\varnothing$, and condition on the initial frog
  taking its first step to $v_0$. By symmetry of the tree, it suffices to prove
  the corollary under this assumption.
  Note that the children
  of the root other than $v_0$ are never visited, since frogs are frozen when they visit the root
  in the self-similar frog model.
  Thus, our goal is to show that at least half the leaves descending
  from $v_0$ are visited in $Ck$ steps with probability at least $p$, for some $p$ to be determined.
  
  Let $v_{-1},v_0,\ldots,v_{k}$ be the path from $\varnothing=v_{-1}$ to an arbitrary leaf 
  $v_k$ descending from $v_0$.
  We will show that $v_k$ fails to be visited
  in $C k$ steps with probability at most $q=q(\beta,C)$, and that $q$ can be made arbitrarily
  small by choosing $\beta$ and $C$ large enough.
  Then, the expected number of leaves descending from $v_0$ that are
  not visited in time $C k$ is at most $qd^{k}$.
  The lemma then follows by applying
  Markov's inequality to show that the number of unvisited leaves is larger than $d^{k}/2$ 
  with probability at most $2q$.
  
  For $i\geq 1$, let $\tau_i'$ be the number of steps after $v_{i-1}$ is visited that
  $v_i$ is first visited. If $v_i$ is never visited, set $\tau_i'=\infty$.
  Now, we observe that the self-similar frog model on $\TT^{k+1}_d$ with frogs frozen
  at leaves is identical to the first $k+1$ levels of the self-similar frog model on
  $\TT_d$ with frogs frozen at level~$k+1$. This frog model can naturally be coupled
  with the self-similar frog model on $\TT_d$ with no freezing. Putting these together yields
  a coupling between $(\tau_i')_{i=1}^{k+1}$ and the random variables $(\tau_i)_{i=1}^{\infty}$
  defined in \thref{lem:speed}.
  
  We claim that if
  \begin{align}\label{eq:tau.requirement}
    \tau_i \leq 2(k-i+1)\qquad\text{for all $1\leq i\leq k$},
  \end{align}
  then $\tau_i'=\tau_i$ for all $1\leq i\leq k$.
  Indeed, suppose that $\tau_i\leq 2(k-i+1)$. From the time $v_{i-1}$ is first visited in the
  self-similar frog model on $\TT_d$, it takes at most $2(k-i+1)$ steps for $v_i$
  to be visited. Since any walk from $v_{i-1}$ to level~$k+1$ back to $v_i$ has length
  at least $2(k-i+1)+1$, this visit to $v_i$ still occurs when frogs are frozen at level~$k+1$.
  Under the coupling, we then have $\tau_i'=\tau_i$ for all $1\leq i\leq k$ as desired.
  
  Thus, if \eqref{eq:tau.requirement} holds and $1+\sum_{i=1}^k\tau_i\leq Ck$, then 
  $v_k$ is woken in $Ck$ steps in the self-similar model on $\TT^{k+1}_d$. Therefore,
  \begin{align*}
    q(\beta,C) &\leq \P\bigl[\text{$\tau_i>2(k-i+1)$ for some $1\leq i\leq k$}\bigr]
       + \P\Biggl[ 1+\sum_{i=1}^k\tau_i > Ck \Biggr].
  \end{align*}
  We now bound the two terms on the right-hand side of this inequality.
  By \thref{lem:speed},
  \begin{align*}
    \P\bigl[\text{$\tau_i>2(k-i+1)$ for some $1\leq i\leq k$}\bigr]
      &\leq \sum_{i=1}^k e^{-\beta(k-i+1)}\leq \frac{e^{-\beta}}{1-e^{-\beta}},
  \end{align*}
  which can be made as small as desired by increasing $\beta$.
  
  By \thref{lem:speed}, the random variables $(\tau_i)$ are independent, and 
  $(\tau_i+1)/2$ is stochastically dominated by the geometric distribution
  on $\{1,2,\ldots\}$ with parameter $1-e^{-\beta}$.  
  By \thref{prop:geo.bound},
  \begin{align*}
    \P\Biggl[ 1+\sum_{i=1}^k\tau_i > Ck \Biggr]
      &\leq \exp\biggl[ -k\biggl(\frac{C(1-e^{-\beta})}{2}-1\biggr)\biggr]
      \leq \exp\biggl( -\frac{C(1-e^{-\beta})}{2}+1\biggr).
  \end{align*}
  For any given $\beta>0$, this can be made arbitrarily small by increasing $C$.
  This proves that $q(\beta,C)$ is as small as desired for large enough $\beta$ and $C$.
  In particular, plugging in numbers, we see that $q(\beta,C)< 1/4$ if
  $\beta\geq 2$ and $C\geq 8$. As $2q$ is the bound on the probability of fewer than
  $d^k/2$ leaves being visited, this completes the proof.
\end{proof}

\section{Fast cover time for large $\mu$}
\label{sec:upper}
We now present the most general version of the cover time upper bound.

\begin{thm} \thlabel{thm:upper}
  Let $\Cc$ be the cover time for the frog model on $\TT_d^n$ with initial frog counts given by 
  $(\eta(v))_{v\in\TT_d^n\setminus\{\emptyset\}}$. Suppose that
  $\eta(v)\pgfsucc\Poi(\mu)$ for all $v$, 
  where $\mu>\beta_0d^2$ for a sufficiently large absolute constant $\beta_0$.
  There exist constants $C_d$ and 
  $n_0(\mu,d)$ such that 
  \begin{align*}
    \prob[\cover >C_d n \log n/\mu] \leq d^{-n}
  \end{align*}
  for all $n\geq n_0(\mu,d)$.
\end{thm}
See Appendix~\ref{sec:comparison} for the definition of the pgf stochastic order denoted
by $\pgfsucc$.
Loosely speaking, the condition $\eta(v)\pgfsucc\Poi(\mu)$ means that the distribution of $\eta(v)$
is larger and more concentrated than the distribution of $\Poi(\mu)$. In particular, if $\eta(v)=k$
deterministically for some integer $k\geq\mu$, then $\eta(v)\pgfsucc\Poi(\mu)$. Thus, our 
result holds for the frog model with $k$ frogs per vertex for $k>\beta_0d^2$.

This theorem follows from two propositions that we explain now. Fix a leaf $v_0 \in \mathbb T_d^n$. 
Label the path from $v_0$ to the root by $v_0, \dots ,v_n = \emptyset$.
In general, we will take $\mu=(3+\beta)d(d+1)$ for some parameter $\beta$, a convenient
form for applying Corollaries~\ref{cor:strong.finite} and \ref{lem:pos_frac}.
The vertex~$v_J$, where
\begin{align}\label{eq:Jdef}
  J=J(d,n,\beta)=\floor{ \log_dn +\log_d( \log n) +5\log_d 10-\log_d\beta},
\end{align}
is far enough from $v_0$ 
that we can show that many frogs visit $v_J$ in time
$O(n\log n)$.
It is also close enough to $v_0$ that one of these frogs at $v_J$
will visit $v_0$ in $O(n \log n)$ steps with high probability (see Figure \ref{fig:upper}).
 These two statements are the content
of \thref{prop:TJ,prop:cleanup}, which we show under Poisson
initial conditions. We complete the proof by applying \thref{cor:comparison} to relax this assumption.

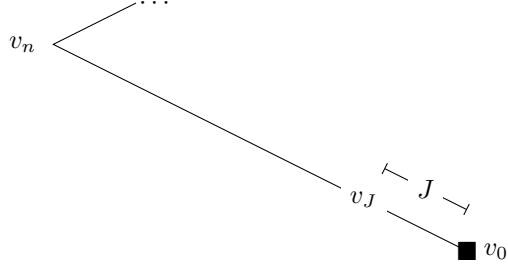
\begin{figure}
\begin{tikzpicture}[scale = .55]
	\node at (2.5,1) {$\cdots$};
	\draw (2,1) -- (0,0) --  node [near end,fill=white] {$v_J$} (10,-5); 
	\draw[|-|] (10,-4) -- node [midway,fill=white] {$J$} (8,-3);
	\node[left = .1 cm] at (0,0) {$v_n$};
	\node[right = .1 cm]  at (10,-5) {$v_0$};
	\node at (10,-5) {$\blacksquare$};
\end{tikzpicture}	
\caption{The basic idea of \thref{thm:upper} is to show that many frogs visit $v_J$ after $O(n\log n)$ 
steps for $J\approx \log_d n + \log_d( \log n) + C$. Once enough frogs are built up at $v_J$,
one of them will visit the leaf $v_0$ with high probability within
$O(n\log n)$ steps.}\label{fig:upper}
\end{figure}

Recall the definition of a stopped version of a frog model from Section~\ref{sec:modified.frog.models}.
\begin{prop} \thlabel{prop:TJ} 
  For some constants $\beta_0$ and $C_d$, the following holds.
  Let $\mu=(3+\beta)d(d+1)$ for $\beta>\beta_0$.
  There exists a stopped version of the frog model on $\TT^n_d$ with i.i.d.-$\Poi(\mu)$ initial conditions
  such that $10n\log n$ frogs have been stopped at vertex~$v_J$
  by time $C_dn\log n/\beta$ with probability at least $1-d^{-3n}$ for all $n\geq n_0(\beta,d)$,
  for some constant $n_0(\beta,d)$.
\end{prop}
\begin{prop} \thlabel{prop:cleanup} 
  Suppose that $10n\log n$ simple random walks start at vertex~$v_J$ in $\TT^n_d$ and move independently,
  and that $n\geq n_0$ for some sufficiently large absolute constant $n_0$.
  For some absolute constant $C$, one of the walks visits $v_0$ within $Cn\log n/\beta$ steps
  with probability $1-d^{-3n}$.
\end{prop}

The upper bound on the cover time follows from \thref{prop:TJ} and \thref{prop:cleanup}.

\begin{proof}[Proof of \thref{thm:upper}]
  First, assume that the sleeping frog counts $(\eta(v))_v$ are i.i.d.-$\Poi(\mu)$.
By  \thref{prop:TJ}, there is a stopped version of the frog model where
$10n\log n$ frogs accumulate at $v_J$ by time $C_dn\log n/\beta$ with probability at least $1 - d^{-3n}$.
At time $\floor{C_dn\log n/\beta}$, unfreeze all frogs and let them resume their simple random
walks.
By \thref{prop:cleanup}, the vertex~$v_0$ is visited in this modified process
by time $C'_dn\log n/\beta$ with probability $1-2d^{-3n}$ for some constant $C'_d$.
If this holds in this stopped and restarted frog model, then it holds in the original frog model as well,
by an obvious coupling.
As $v_0$ was arbitrary, each leaf is visited with probability at least $1-2d^{-3n}$, and
the expected number of leaves unvisited by time $C'_dn\log n/\beta$ is therefore at most
$2d^{-2n}$. 

Now, we extend this to non-Poisson initial conditions.
Let $N$ be the number of leaves visited by time $C'_dn\log n/\beta$ in the Poisson frog model,
which we have shown to satisfy
\begin{align*}
  \E N \geq d^n - 2d^{-2n}.
\end{align*}
Let $N'$ be the corresponding count of visited leaves for the frog model defined in the statement
of this theorem. By \thref{cor:comparison}, we have $\E N'\geq \E N$. Thus, the expected
number of unvisited leaves in this frog model is also at most $2d^{-2n}$, and by Markov's inequality
there is an unvisited leaf with probability at most $2d^{-2n}\leq d^{-n}$.
Once all leaves are visited, all vertices of the tree have been visited,
completing the proof.
\end{proof}

%

\subsection{Establishing \thref{prop:TJ}}
\label{sec:TJ}

The goal of this section is to prove it overwhelmingly likely that $\Omega(n\log n)$
frogs accumulate at $v_J$ in time $O(n\log n/\beta)$, recalling the definitions
of $v_0,\ldots,v_n$, $J$, and $\beta$ from the beginning of the section.
Our argument is sequential: we show that many frogs flow from $v_n$ to $v_{n-1}$, which
spurs many frogs to flow into $v_{n-2}$, and so on. To make this precise, we introduce
random variables $I_k$ for $J+1\leq k\leq n$. Loosely speaking,
$I_k$ is the quantity of frogs that must start at $v_k$ so that frogs flow
steadily to $v_{k-1}$ at a rate of $\Omega(\beta)$ per time step.
Now, imagine running the frog model until $I_{n-1}$ frogs have built up
at $v_{n-1}$. Once this happens, frogs will flow steadily to $v_{n-2}$; allow them
to build up until there are $I_{n-2}$ there, which will take time $O(I_{n-2}/\beta)$.
Continuing in this way, we build up $I_{J+1}$ frogs at $v_{J+1}$ in time
\begin{align*}
  O\Bigl(\textstyle \sum_{k=J+1}^{n-1} I_k/\beta\Bigr),
\end{align*}
plus the time to get the first $I_{n-1}$ frogs to $v_{n-1}$. This creates a steady flow of frogs
to $v_J$, and after another $O(n\log n/\beta)$ steps, we have
produced $\Omega(n\log n)$ visits to $v_J$. Thus, the main task is to show
that $\sum_{k=J+1}^{n-1} I_k/\beta$ is unlikely to be large. We do this by showing an exponential
tail bound for $I_k$, from which it follows that it is exponentially likely that this sum is $O(n)$.

We mention that we use nonbacktracking frogs throughout this section.
This coordinates well with our results regarding the self-similar frog model in Section~\ref{sec:exterior}.
Only at the very end will we apply \thref{cor:time.change} to move our results 
back to the usual frog model.

\subsubsection{Definition of $I_k$}
We first define a family
of processes $\FM(v_k,\ell)$, which are frog models limited to the subtree
$\TT_d^n(v_k)$ with an extra $\ell$
frogs initially at $v_k$. Then we define $I_k$ as the smallest
$\ell$ for which $\FM(v_k,\ell)$ produces a steady stream of frogs entering $v_{k-1}$:

\begin{define}[$\FM(v_k,\proc)$ and $I_k$]\thlabel{def:I}
  Let $\mu = (3+\beta)d(d+1)$.
  For $J<k < n$ and $\ell\geq 1$,
  let $\FM(v_k,\ell)$ be a frog model defined as follows. 
  We place sleeping frogs only within $\TT_d^n(v_k)\setminus\TT_d^n(v_{k-1})$.
  At all of these vertices except for $v_k$, place $\Poi(\mu)$ frogs
  per site as usual.  At $v_k$ itself, we place $\Poi(\mu)$ frogs
  plus an extra $\ell$ \emph{special frogs}, as we will call them.
  The paths of the special frogs are root-biased nonbacktracking walks
  stopped at $v_{k-1}$ and $v_{k+1}$ with their first steps conditioned to move to a descendant of
  $v_k$ (that is, to move away from $v_{k+1}$). The paths of all other frogs are root-biased nonbacktracking
  walks stopped at $v_{k-1}$ and $v_{k+1}$. Vertex~$v_k$ is the starting vertex for the process; all frogs
  there are initially awake.
  
  For a fixed value of $k$, we consider $\FM(v_k,\ell)$ to be coupled for all choices of $\ell$
  in the natural way. That is, we suppose that there is an infinite pile of special frogs at $v_k$,
  and $\FM(v_k,\ell)$ uses only the first $\ell$ of them. We denote the collection of coupled
  frog models $(\FM(v_k,\ell))_{\ell\geq 1}$ by $\FM(v_k,\proc)$.
  
  For $J<k< n$, we define the random variable $I_k$ to be the smallest integer $\ell$ 
  such that the number of frogs frozen at $v_{k-1}$ by time~$t$ in 
  $\FM(v_k,\ell)$ is at least $\beta t/10000$ for all $\max(3,\ceil{\ell/\beta})\leq t\leq d^k$.
  Observe that this test becomes vacuous when $\ell>\beta d^k$, and therefore $I_k\leq \beta d^k+1$.
\end{define}

As we have remarked, one should think of $I_k$ as the minimum number of special frogs
at $v_k$ to ensure a steady flow of frogs into $v_{k-1}$. This ``steady flow'' is at
rate $\Omega(\beta)$ per time step. For technical reasons, we only require it to start
at time $\max(3,\ceil{\ell/\beta})$. We require the flow to continue only up to time $d^k$
because it is impossible for it to continue much longer, since there are
only $O(\beta d^k)$ frogs in the entire system $\FM(v_k,\ell)$.

\subsubsection{Exponential tail bound for $I_k$}
The bulk of our work in Section~\ref{sec:upper} is to prove the following exponential tail bound on $I_k$:
\begin{prop} \thlabel{prop:I_k.tail}
  For some constants $c,C>0$, the following holds.
  Let $\mu =(3+\beta)d(d+1)$. For $\beta\geq 10000$, 
  it holds for any integers $J <k < n$ and $1\leq\ell\leq d^k$ that
\begin{align}\label{eq:I_k.tail}
  \P[I_k >  \ell] \leq Ce^{-c\ell}.
\end{align}
\end{prop}
Once this is proven, a short argument shows that $\sum_{k=J+1}^{n-1} I_k=O(n)$ is exponentially likely if
the random variables
$(I_k)$ are assumed to be i.i.d. This is the most important element of the proof of \thref{prop:TJ}.
To prove \thref{prop:I_k.tail}, we must argue that $\FM(v_k,\ell)$ is exponentially likely in $\ell$
to send a steady flow of frogs to $v_{k-1}$.
There are two parts to this argument.
\label{page:time.segments}
From times $\max(3,\ceil{\ell/\beta})$ to $k$, we obtain the necessary quantity of frogs at $v_{k-1}$
as a direct consequence of \thref{cor:strong.finite} (see \thref{lem:stretch.2}).
To show that the flow condition is maintained beyond this, we leverage \thref{lem:pos_frac}
to prove it
exponentially likely in $\ell$ that we wake up a positive fraction of all frogs in $\TT_d^n(v_k)$
by time $O(k)$. We then show that enough of these frogs frogs will move to $v_{k-1}$
to give us our steady flow of frogs from time $14k$ to $d^k$ (see \thref{lem:stretch.4}).
To bridge the gap between times $k$ and $14k$, we make $\beta$ large enough to build up a surplus of frogs 
at $v_{k-1}$ during the first $k$ time steps. This ensures that the steady flow requirement is met until time $14k$ even if no additional frogs visit $v_{k-1}$ for times between $k$ and $14k$ (see \thref{lem:stretch.3}).

We now begin working towards \thref{prop:I_k.tail}. We start with two technical estimates.
First, we show that a frog at a leaf of $\TT^k_d$ hits the root in $t$ steps
with probability $\Omega(td^{-k})$.
\begin{lemma}\thlabel{lem:root.miss}
  Consider a root-biased nonbacktracking random walk on $\TT^{k}_d$ starting from a leaf for $k\geq 2$.
  For any integer $k+2\leq t \leq d^k$,
  the walk visits the root in its first $t$ steps with probability at least
  $(t-k-2)d^{-k}/4$.
\end{lemma}
\begin{proof}
  Let $T$ be the first time that the walk hits the root.
  We decompose the walk into a sequence of independent excursions from the leaves.
  Since each excursion reaches the root with probability $d^{-k+1}$, the number of unsuccessful excursions
  before hitting the root is $\Geo(d^{-k+1})$. Each unsuccessful excursion has length distributed as
  $2\widetilde{G}$, where $\widetilde{G}$ is a geometric random variable on 
  $\{1,2,\ldots\}$ with parameter $(d-1)/d$
  conditioned to be less than $k$. Let $\widetilde{G}^{(i)}$ be independent copies of $\widetilde{G}$,
  and let $G^{(i)}$ be independent and distributed as the unconditioned geometric distribution on
  $\{1,2,\ldots\}$ with the same parameter. Thus,
  \begin{align*}
    T &\eqd k + 2\sum_{i=1}^{\Geo(d^{-k+1})}\widetilde{G}^{(i)}
      \preceq
      k + 2\sum_{i=1}^{1+\Geo(d^{-k+1})}G^{(i)}
      \eqd k + 2\bigl(1 + \Geo\bigl((d-1)d^{-k}\bigr)\bigr),
  \end{align*}
  with the last step using the fact that the sum of $1 + \Geo(p)$ many independent $1 + \Geo(q)$
  random variables is a $1 + \Geo(pq)$ random variable.
  Therefore,
  \begin{align*}
    \P[T\leq t] &\geq \P\biggl[1 + \Geo\bigl((d-1)d^{-k}\bigr) \leq \frac{t-k}{2}\biggr]\\
    &= 1 - \bigl(1 - (d-1)d^{-k}\bigr)^{\floor{(t-k)/2}}
    \geq 1 - \exp\biggl( -\frac{(t-k-2)d^{-k}}{2} \biggr).
  \end{align*}
  Using the bound $1-e^{-x}\geq x/2$ for $x\in[0,1]$ along with the assumption that
  $t\leq d^k$,
  \begin{align*}
    \P[T\leq t] &\geq \frac{(t-k-2)d^{-k}}{4}.\qedhere
  \end{align*}
\end{proof}

\begin{lemma}\thlabel{lem:balls.bins}
  Suppose that $m$ balls are placed uniformly and independently into $n$ bins, with $m\geq 3n$.
  Let $Z$ be the number of occupied bins. Then
  \begin{align*}
    \P[ Z\leq 2n/3] &\leq e^{-m/54}.
  \end{align*}
\end{lemma}
\begin{proof}
  Imagine that we place the balls one after another, and
  define $Z_i$ as the number of bins occupied after $i$ balls have been placed. Let
  $T=\min\{i\colon Z_i\geq 2n/3\}$. We need to bound the probability that $T>m$.
  We observe that $(Z_i)_{i\geq 0}$ is a pure birth process
  with $\P[Z_{i+1}=Z_i+1\mid Z_i] = 1-Z_i/n$ and $Z_0=0$.
  Let $(Y_i)_{i\geq 0}$ be a pure birth process starting at $0$ and increasing at each step with
  probability~$1/3$. We can couple the two processes so that $(Y_i)$ increases only
  when $(Z_i)$ does up to time~$T$. 
  We apply \thref{prop:Poi.bound} to the random variable $Y_m$, which is distributed as
  $\Bin(m,1/3)$, and we get
  \begin{align*}
    \P[T>m] \leq \P[Y_m\leq2n/3] &\leq \exp\biggl(-\frac{(1-2n/m)^2m}{6}\biggr)
      \leq e^{-m/54},
  \end{align*}
  using our assumption $m\geq 3n$.
\end{proof}

\newcommand{\Csus}{12}

We are now ready to start on the proof of \thref{prop:I_k.tail}.
Let $X_t$ be the number of frogs frozen at $v_{k-1}$ by time~$t$ in $\FM(v_k,\ell)$.
The basic idea is that if $I_k>\ell$, then $X_t<\beta t/10000$ occurs for some 
$\max(3,\ceil{\ell/\beta})\leq t\leq d^k$.
Thus it suffices to show that the probability of this event decays
exponentially in $\ell$.
In the next three lemmas, we break the time interval $\max(3,\ceil{\ell/\beta})\leq t\leq d^k$ into the
three segments described on page~\pageref{page:time.segments},
and we bound the probability that $X_t<\beta t/10000$ on any of them.

The first time segment is for length~$k$, which is the height of the tree rooted at $v_k$. 
As we mentioned, we use the application of strong recurrence to the finite tree in \thref{cor:strong.finite} 
to accrue $\Omega(k)$ frogs at $v_j$ in time $k$. 
  
  \begin{lemma}\thlabel{lem:stretch.2}
    With the conditions of \thref{prop:I_k.tail},
    \begin{align*}
      \P\bigl[ \text{$X_t<\beta t/10000$ for some $\max(3,\ceil{\ell/\beta})\leq t \leq k$}\bigr] &\leq
      Ce^{-c\ell}
    \end{align*}
    for some constants $c,C>0$.
  \end{lemma}
  \begin{proof}
    With probability $1-d^{-\ell}$, some child of $v_k$ other than $v_{k-1}$ is visited
    by one of the $\ell$ special initial frogs on the first step of $\FM(v_k,\ell)$. 
    Call this event
    $E$. Conditional on $E$,
    let $u$ be a visited child, and couple $\FM(v_k,\ell)$ with
    the self-similar frog model on $\TT^k_d$ with frogs frozen at the leaves,
    as follows.
    Identify $v_k$ with the root of the self-similar model. Let $u'$ be the child
    of the root in $\TT^k_d$ first visited by the initial frog in the self-similar model.
    Identify $\TT^n_d(u)$
    in $\FM(v_k,\ell)$ with $\TT^k_d(u')$ in the self-similar model.
    Make the number of initial frogs in $\TT^n_d(u)$ in $\FM(v_k,\ell)$ identical to the number of initial
    frogs in $\TT^k_d(u')$ in the self-similar model. Let each of these frogs in $\FM(v,\ell)$
    follow the corresponding frog in the self-similar model until it is frozen. After, each frog in
    $\FM(v,\ell)$ continues as a root-biased nonbacktracking walk independent of the self-similar model.
    Similarly, let the initial frog that moved to $u$ in $\FM(v_k,\ell)$ match the initial frog in the
    self-similar model until it is frozen.
    
    By this coupling and \thref{cor:strong.finite}, the count of frogs moving from $u$ back to $v_k$
    by time $2t$ is stochastically
    at least $\Poi(\beta d t)$ for any integer $1\leq t\leq k-1$, conditional on $E$. 
    As each of these frogs
    moves next to $v_{k-1}$ with probability $1/d$, we have
    $X_{2t+1}\succeq \Poi(\beta t)$.
    Restating this, conditional on $E$,
    \begin{align}\label{eq:Xt.bound}
      X_t\succeq \Poi\Bigl(\ceil{(t-1)/2}\beta\Bigr)\succeq \Poi\biggl(\Bigr(\frac{\beta}{2}-1\Bigr)t\biggr)
    \end{align}
    for any $3\leq t\leq 2k-1$.
    Let $t_0=\max(3,\ceil{\ell/\beta})$.
    By \thref{lem:Poi.sequence},
    \begin{align}\label{eq:lcl}
      \P\bigl[ \text{$X_t<\beta t/10000$ for any $t_0\leq t\leq k$} \bigmid E \bigr]
        &\leq 2e^{-\Omega(\beta\ceil{\ell/\beta})}\leq 2e^{-\Omega(\ell)}.
    \end{align}
    Note that for simplicity we have limited the range to $t_0\leq t\leq k$,
    even though \eqref{eq:lcl}  holds for a larger interval.
    Combined with $\P[E]=1-d^{-\ell}$, this proves the claim.   
  \end{proof}
   
Our next time segment is from $k$ to $14k$, bridging the gap between our first and third segments. 
The argument here is rather simple: In proving \thref{lem:stretch.2}, we
built up sufficiently many frogs at time $k$ to keep $X_t$ large enough until time $14k$. 

  \begin{lemma}\thlabel{lem:stretch.3}
    With the conditions of \thref{prop:I_k.tail},
    \begin{align*}
      \P\bigl[ \text{$X_t<\beta t/10000$ for some $\max(\ceil{\ell/\beta},k)< t \leq 14k$}\bigr] &\leq
        e^{-c\ell}
    \end{align*}
    for some constant~$c>0$.
  \end{lemma}
  \begin{proof}
    From \eqref{eq:Xt.bound} in the previous proof,
    \begin{align*}
      X_k\succeq \Poi\biggl(\Bigl(\frac{\beta}{2}-1\Bigr)k\biggr).
    \end{align*}
    \thref{prop:Poi.bound} then gives
    \begin{align*}
      \P\biggl[X_k<\frac{14\beta k}{10000}\biggr] &\leq e^{-\Omega(\beta k)}\leq e^{-\Omega(\ell)}
    \end{align*}
    if $14k\geq\ceil{\ell/\beta}$, which we can assume since otherwise the lemma is vacuous.
    This completes the proof, since if $X_t<\beta t/10000$ for any $k\leq t\leq 14k$, 
    then $X_k < 14\beta k/10000$.  \end{proof}

The last segment of time is the largest, from $14k$ to $d^k$.
The idea is to wake a large fraction of the leaves of $\mathbb T_d(v_k)$ and show that this produces
a steady stream of frogs to $v_{k-1}$ up to time $d^k$. \thref{lem:pos_frac} ensures that with positive probability, the self-similar frog model with a single initial frog wakes a positive fraction of the leaves.
Essentially, we need to show that if we start the process with $\ell$ frogs active at the root, then 
the chance of waking a positive fraction of the leaves improves exponentially in $\ell$.

The trick to doing so is to find many independent opportunities 
to apply \thref{lem:pos_frac}, so that we may boost the fixed probability bound to an exponential one.
We start by letting the $\ell$ initial frogs in $\FM(v_k,\ell)$ move a distance of $L=\floor{\log_d \ell /3}$
down the tree. By a comparison to placing balls uniformly into bins, we show that these frogs
are exponentially likely in $\ell$ to cover at least $2/3$ of the vertices at this level.
We then apply \thref{lem:pos_frac} to the subtrees rooted at the visited vertices to show that
each independently has at least probability $1/2$ of having half its leaves wake in time $O(k)$.
Since there are $\Omega(\ell)$ of these subtrees, it is exponentially likely in $\ell$
that this occurs for a positive fraction of them. All together, this demonstrates that
it is exponentially likely in $\ell$ that a positive fraction of leaves of $\TT_d^n(v_k)$ are
woken in time $O(k)$. With this many frogs awake, standard hitting estimates from a leaf to a root give us a steady flow of frogs to $v_k$ up to time $d^k$. We now make this outline precise.

  \begin{lemma} \thlabel{lem:stretch.4}
    With the assumptions of \thref{prop:I_k.tail},
    \begin{align}\label{eq:stretch.4}
      \P\biggl[ \text{$X_t<\frac{\beta t}{10000}$ for some $\max\bigl(\ceil{\ell/\beta},\,14k\bigr)
                  \leq t\leq d^k$}\biggr]
     &\leq Ce^{-c\ell}
    \end{align}
    for some constants $c,C>0$.
  \end{lemma}
  \begin{proof}
      We can assume without loss of generality that $\ell\geq 3$, since the $\ell=1,2$ cases can be
    made trivial by choosing $C$ large enough.
      Let $L = \floor{\log_d(\ell/3)}\geq 0$.
  Let $\Vv_L$ and $\Vv_{L+1}$ be respectively 
  the sets of distance~$L$ and $L+1$ descendants of $v_k$ that are not descendants of $v_{k-1}$.
  The restriction $\ell\leq d^k$ implies that $v_k$ has at least $L+1$ generations
  of descendants.
  
  For each $v\in\Vv_{L+1}$, couple a self-similar frog model on $\TT^{k-L}$ with $\FM(v_k,\ell)$ in the same
  way as in \thref{lem:stretch.2}. This time, if $u'$ is the child of the root in $\TT^{k-L}_d$ first
  visited by the initial frog, then $\TT^{k-L}_d(u')$ is identified with $\TT^n_d(v)$, 
  and the root of $\TT^{k-L}_d$ is identified with the parent of $v$. If $v$ is ever visited
  in $\FM(v_k,\ell)$, then choose one of its activators and match its path up with the initial frog
  in the self-similar model. All other aspects of the coupling are as in \thref{lem:stretch.2}.
  Note that under this coupling, the self-similar models matched for each $v\in\Vv_{L+1}$ are independent.
  
  For $v\in\Vv_L$, let $A_v$ be an indicator on $v$ being visited by one of the $\ell$ initial frogs
  in $\FM(v_k,\ell)$, and let $A=\sum_{v\in\Vv_L}A_v$. 
  The total number of vertices at level~$L$ is $d^L\leq\ell/3$, and each
  initial frog is equally likely to go to any of them.  
  By \thref{lem:balls.bins}, at least $2/3$ of these vertices are visited with probability
  $1-e^{-\ell/54}$. On this event, $A\geq (2/3 - 1/d)d^L\geq d^L/6$. Note that all of this holds
  even when $L=0$, when $A=A_{v_k}=1$ deterministically.
  
  Now, condition on $(A_v)_{v\in\Vv_L}$.
  For every child~$u$ of a vertex $v\in\Vv_L$ satisfying $A_v=1$, 
  let $B_u$ be an indicator on some frog woken at $v$ moving immediately to $u$.
  By Poisson thinning, there are independently $\Poi\bigl((3+\beta)d\bigr)$ frogs
  woken at $v$ moving to $u$. Hence, conditional on $(A_v)_{v\in\Vv_L}$, the random variables
  $B_u$ for such $u$ as described above are i.i.d.-$\Ber\bigl(e^{-(3+\beta)d}\bigr)$.
  
  Call $u\in\Vv_{L+1}$ \emph{sustaining} if in the self-similar model coupled to it, at least
  $d^{k-L-1}/2$ leaves are activated in $\Csus( k-L-1)$ steps.
  Let $S_u$ be an indicator on $u$ being sustaining. The random variables $(S_u)_{u\in\Vv_{L+1}}$
  are independent of each other and of all random variables $A_v$ and $B_u$ defined in the previous
  paragraph. Let
  \begin{align*}
    S = \sum_{\substack{\text{$u$ child of $v$}\\v\in\Vv_L,\,A_v=1}} B_uS_u.
  \end{align*}
  Conditional on $(A_v)_{v\in\Vv_L}$, the random variable $S$ is the sum of independent indicators,
  and $\E[ S\mid A]\geq e^{-(3+\beta)d}Ad/2\geq .48Ad$ by \thref{lem:pos_frac}. Conditional on $A\geq d^L/6$,
  we then have $S\geq d^{L+1}/25$ with probability at least $1-e^{-\Omega(d^{L+1})}\geq 1-e^{-\Omega(\ell)}$.
  
  Now, we claim that if $S>d^{L+1}/25$, it is unlikely that $X_t<\beta t/10000$ for any
  $\ceil{\ell/\beta}\leq t\leq d^k$. On the event $\{S>d^{L+1}/25\}$, there are stochastically at least
  $\Poi(\mu d^k/50)$ frogs activated by time $L+\Csus (k-L-1)$ at leaves descending from $v_0$.
  As the paths of the frogs at the leaves are independent of $S$, conditional on $S>d^{L+1}/25$
  their paths remain independent root-biased nonbacktracking walks. By \thref{lem:root.miss},
  the number of these frogs that have visited $v_k$ by time $L+\Csus( k-L-1) + (t+k+2)$ 
  is stochastically at least $\Poi(t\mu/200)$  
  for any $0\leq t \leq d^k-k-2$. We then thin by $1/d$ to get the number of frogs
  frozen at $v_{k-1}$ after one more step. Hence,
  \begin{align*}
    X_{13 k + t} \geq X_{L+12 (k-L-1) + t + k+3}&\succeq
      \Poi\biggl(\frac{t\mu }{200d}\biggr)
      =\Poi\biggl(\frac{t(3+\beta)(d+1) }{200} \biggr).
  \end{align*}
  For $t\geq k$,
  \begin{align*}
    \frac{t(3+\beta)(d+1) }{200} \geq \frac{\beta (13k+t)}{1000}.
  \end{align*}
  By \thref{lem:Poi.sequence},
  \begin{align*}
    \P\biggl[ \text{$X_t<\frac{\beta t}{10000}$ for some 
                        $\max\bigl(\ceil{\ell/\beta},\,14k\bigr)< t\leq d^k$}
                        \biggmid S > \frac{d^{L+1}}{25} \biggr] 
           &\leq 2 e^{-\Omega\beta\ceil{\ell/\beta}}
           \leq 2e^{-\Omega(\ell)}.
  \end{align*}
  Combined with the estimates on $\P[S > d^{L+1}/25\mid A\geq d^L/6]$
  and on $\P[A\geq d^L/6]$, this completes the proof.
  \end{proof}

\begin{proof}[Proof of \thref{prop:I_k.tail}]
  Lemmas~\ref{lem:stretch.2}--\ref{lem:stretch.4} combine via a union bound to
  prove \eqref{eq:I_k.tail}.
\end{proof}

\subsubsection{Final steps toward \thref{prop:TJ}}
We are already done with the hard work toward proving \thref{prop:TJ}.
As we described at the beginning of Section~\ref{sec:TJ}, our argument
requires us to feed $I_{n-1}$ frogs into $v_{n-1}$ to get a steady flow
into $v_{n-2}$, then wait for $I_{n-2}$ frogs to flow into $v_{n-2}$, and so on.
What remains is to show that this happens quickly by stitching together the processes 
$\FM(v_k,\proc)$ and applying \thref{prop:I_k.tail}. In our next lemma, we collect
$I_{n-1}$ frogs at $v_{n-1}$ to set things in motion.

In this section, we have elected to simplify computations by frequent use of big-O
notation. We will be very strict in our use of it: an
expression $O(f)$ or $\Omega( f )$ denotes a quantity bounded respectively
from above or from below by $Cf$, where $0<C<\infty$ is an absolute constant
not depending on $d$, $n$, $\mu$, or any other parameter.
For example, the expression $O(Cn/\beta)$ in the next lemma could be replaced
by $C'Cn/\beta$, where $C'$ is an absolute constant with no dependence on $d$, $n$, $\beta$,
or $C$.

\begin{lemma}\thlabel{lem:initial.frogs}
  Consider the nonbacktracking frog model on $\TT^n_d$ with i.i.d.-$\Poi(\mu)$ frogs per site
  where $\mu = (3+\beta)d(d+1)$. Given $C>0$, there exists $\beta_0=\beta_0(C)$ such that
  for $\beta\geq\beta_0$, there is a stopped version of the frog model with the following
  property: it holds with probability at least $1-e^{-Cdn}$ that
  at least $Cdn$ frogs whose last step was from $v_n=\emptyset$
  are frozen at $v_{n-1}$ by time $O(Cn/\beta)$ for $n\geq n_0(C,\beta,d)$.
\end{lemma}
\begin{proof}
  Suppose that some child $u$ of the root is visited at time~$t$ for the first time.
  We first mention that we can couple the frog model restricted to $\{\emptyset\}\cup\TT_d^n(u)$
  from time~$t$ on
  with the self-similar frog model on $\TT^n_d$ with frogs frozen at leaves from time~$1$ on, 
  as we have often done before: 
  simply have all frog paths identical in both models up until time a frog is stopped
  in the self-similar model. By this coupling and \thref{cor:strong.finite}, our original
  frog model has stochastically at least $\Poi(c d n)$ visits to $\emptyset$
  from $u$ by time $t+2cn/\beta$, assuming that $\beta$ is large enough that
  $2cn/\beta\leq n-1$.
  
  We now apply this fact repeatedly to prove the lemma.
  Let $\emptyset'$ be the child of the root first visited by the initial frog.
  The gist of the argument is to couple the frog model on $\emptyset\cup\TT^n_d(\emptyset')$ 
  with the self-similar model as above to obtain $\Poi(cdn)$ visits to $\emptyset$ in time
  $2cn/\beta$.
  From this, we are very likely to visit, say, one third of the children of the root
  by time $2cn/\beta+1$. For each visited child $v$, we couple
  the frog model on $\{\emptyset\}\cup\TT^n_d(v)$ with the self-similar frog model
  to get another $\Poi(cdn)$ visits to the root after another $2cn/\beta$ steps. 
  Summing the contributions from
  all $\Omega(d)$ visited children, we have $\Poi(cd^2n)$ visits to the root, and after one
  more step we have $\Omega(dn)$ frogs at $v_{n-1}$. We will write out this argument
  with all details below, but we remark that the details are less enlightening
  than the description we have just given.
  
  We do the argument first in the $d\geq 3$ case.
  Let $c>1$ be a large constant, to be specified in more detail later. 
  In this argument, we use the phrase \emph{with overwhelming probability} to mean
  with probability at least $1-e^{-\Omega(cdn)}$ for sufficiently large $n$ (where the meaning
  of sufficiently large can depend on $c$, $d$, and $\beta$). 
  Each instance of the phrase might have a different
  constant in the $\Omega(cdn)$ expression.
  Observe that by a union bound, an intersection
  of a bounded number of events holding with overwhelming probability also holds
  with overwhelming probability.
  
  Choose $\beta_0$
  large enough that $2cn/\beta_0\leq n-1$, and assume that $\beta\geq\beta_0$.
  We then have stochastically at least $\Poi(cdn)$ visits from $\emptyset'$ to $\emptyset$ 
  in time $2cn/\beta$ by the coupling described above. Each frog that moves from $\emptyset'$
  to $\emptyset$ moves next outside of $\{\emptyset,v_{n-1}\}$
  with probability at least $1-(d+2)/d^2=\Omega(1)$, recalling the dynamics of root-biased
  nonbacktracki ng walk from Section~\ref{sec:modified.frog.models}.
  Thus, by time $2cn/\beta+1$, at least $\Poi\bigl(\Omega(cdn)\bigr)$ frogs have done so.
  By \thref{prop:Poi.bound}, this quantity of frogs is at least $\Omega(cdn)$ with
  overwhelming probability. Conditional on this occurring, 
  each of these frogs is equally likely to visit
  any of the children of the root other than $\emptyset'$ and $v_{n-1}$.
  By \thref{lem:balls.bins}, the number of these children visited is strictly
  greater than $(d-2)/3$ with overwhelming probability. Conditional on this, for
  each child of the root $v\neq \emptyset',v_{n-1}$ visited, we couple the frog model
  on $\{\emptyset\}\cup\TT^n_d(v)$ with a self-similar model. For each $v$, we then
  obtain $\Poi(cdn)$ visits from $v$ to $\emptyset$ by time $2cn/\beta+1$, giving
  us $\Poi\bigl(\Omega (cd^2n)\bigr)$ such visits in all. Each frog moves next to $v_{n-1}$ with probability
  $(d+1)/d^2$, giving us $\Poi\bigl(\Omega(cdn)\bigr)$ visits to $v_{n-1}$ from $\emptyset$
  in time $2cn/\beta+2$.
  Finally, by \thref{prop:Poi.bound}, this quantity is at least $\Omega(cdn)$ with overwhelming
  probability.
  
  When $d=2$, start the argument the same, obtaining $\Poi(2cn)$ visits from
  $\emptyset'$ to $\emptyset$ by time $cn/\beta$. Depending on whether $\emptyset'=v_{n-1}$,
  each of these frogs moves next to $v_{n-1}$ with probability $3/4$ or $1/4$.
  In either case, we have $\Poi\bigl((\Omega(cn)\bigr)$ frogs moving from $\emptyset$ to $v_{n-1}$
  in time $cn/\beta+1$, and by \thref{prop:Poi.bound}, there are $\Omega(cn)$ of them
  with overwhelming probability.
  
  Thus, in both cases, we have $\Omega(cdn)$ frogs stopped at $v_{n-1}$ after moving there
  from $\emptyset$ in time $O(cn/\beta)$ with overwhelming probability. Choosing $c$ to
  equal $C$ multiplied by a sufficiently large constant completes the proof.
\end{proof}

We now prove the equivalent of \thref{prop:TJ} for the nonbacktracking frog model on $\TT_d^n$.
After this, we will apply \thref{cor:time.change} to transfer the result to the usual frog model.
Recall from \eqref{eq:Jdef} that $J=\floor{ \log_dn +\log_d( \log n) +5\log_d 10-\log_d\beta}$.

\begin{prop}\thlabel{prop:TJ.nonbacktracking}
  Consider the nonbacktracking frog model on $\TT^n_d$ with i.i.d.-$\Poi(\mu)$ initial
  conditions where $\mu=(3+\beta)d(d+1)$.  
  For any constant $C$, for all $\beta\geq\beta_0(C)$ and $n\geq n_0(\beta,d,C)$, 
  there is a stopped version of the model such that
  at least $10n\log n$ frogs are stopped at $v_J$ by time $O(n\log n/\beta)$
  with probability at least $1-e^{-Cdn}$.
\end{prop}

\begin{proof}  
  This proof is somewhat long, but it just stitches together the estimates
  made earlier in the section.
  We start with an informal sketch.
  Start with  the nonbacktracking frog model on $\TT^n_d$ with i.i.d.-$\Poi(2\mu)$ frogs, splitting
  the frogs at each site into two collections of $\Poi(\mu)$ each. 
  With the first collection, we run the frog model
  to accumulate $O(Cdn)$ frogs at $v_{n-1}$, which we can do in time $O(Cn/\beta)$ with overwhelming
  probability by \thref{lem:initial.frogs}.
  We then abandon this first set of frogs and switch to the second,
  giving ourselves a fresh frog model with i.i.d.-$\Poi(\mu)$ frogs per site but with
  an extra $O(Cdn)$ frogs deposited at $v_{n-1}$.
  We can now couple the process with $\FM(v_{n-1},O(Cdn))$.
  Since $I_{n-1}\leq O(Cdn)$ with overwhelming probability
  by \thref{prop:I_k.tail}, a steady stream of frogs flows to $v_{n-2}$.
  When $I_{n-2}$ frogs have built up there, we couple the process to $\FM(v_{n-2},I_{n-2})$,
  and we know that a steady stream of frogs will flow
  to $v_{n-3}$. Continuing in this way, we eventually feed $I_{J+1}$ frogs in $v_{J+1}$, creating a steady
  stream of frogs into $v_J$. After $O(n\log n/\mu)$ steps, enough frogs have built up
  at $v_J$ and we are finished.

  Now, we carry out the details.
  We will be proving our proposition with $\mu$ replaced by $2\mu$,
  which is equivalent by adjusting $\beta_0$.
  We define a process based on the usual frog model with i.i.d.-$\Poi(2\mu)$ frogs per site
  in which frogs are repeatedly stopped and restarted.
  We refer to it as the \emph{slowed process}.
  To define it,
  separate the sleeping frogs in $\TT^n_d$ into two independent $\Poi(\mu)$-distributed
  batches at each vertex. 
  Let the initial frog at the root
  move as usual, as a root-biased nonbacktracking walk.
  For sleeping frogs in the first batch, let their paths be root-biased 
  nonbacktracking walks stopped on moving from the root to $v_{n-1}$.
  Keep all second-batch frogs frozen for now.
  
  Let $\FM(v_k,\proc)$ be independent for all $J+1\leq k< n$. Recall that $I_k$
  is a function of $\FM(v_k,\proc)$, and hence $I_{J+1},\ldots,I_{n-1}$ are independent.
  Once $I_{n-1}$ frogs have been frozen at $v_{n-1}$ in the slowed process, unfreeze all frogs
  accumulated there. Halt all other first-batch frogs at this time and ignore them afterwards.
  
  We now allow the second-batch frogs to work at last.
  When the frogs at $v_{n-1}$
  are unfrozen, couple them with the special frogs in $\FM(v_{n-1},\proc)$.
  Also couple the numbers and paths of second-batch frogs in $\TT_d^n(v_{n-1})\setminus \TT_d^n(v_{n-2})$
  with the normal frogs in $\FM(v_{n-1},\proc)$. Thus, all frogs move (past the first step) as
  nonbacktracking walks frozen at $v_{n-2}$ and $v_n$.
  
  Once $I_{n-2}$ frogs are frozen at $v_{n-2}$, halt all other frogs forevermore, and unfreeze
  these frogs. Couple them and the second-batch frogs in $\TT_d^n(v_{n-2})\setminus\TT_d^n(v_{n-3})$
  with $\FM(v_{n-2},\proc)$ as above. Let all frogs move until $I_{n-3}$ frogs have been frozen
  at $v_{n-3}$. We continue on in this way until $I_{J+1}$ frogs are frozen at $v_{J+1}$.
  We then continue for one last step, unfreezing the frogs at $v_{J+1}$, halting all other ones
  permanently, and coupling the process with $\FM(v_{J+1},\proc)$.
  Finally, we run the process until
  $10n\log n$ frogs are frozen at $v_J$.
  
  We claim that to prove this proposition, it suffices to prove the same
  bound for the slowed process.
  This is intuitively very clear: If we remove all of the stops and restarts
  at vertices other than $v_J$, the resulting model is a stopped version of the
  nonbacktracking frog model. Furthermore, every frog that is stopped at $v_J$
  by a given time in the slowed process will also be stopped at $v_J$ by this time
  in the stopped proess.
  Hence, it suffices to prove that
  at least $10n\log n$ frogs are stopped at $v_J$ in the slowed process at 
  time $O(Cn\log n/\beta)$ with probability at least $1-e^{-Cdn}$.
  
  The rest of the proof is to show this.
  We claim that 
  at least $10n\log n$ frogs are stopped at $v_J$ in
  the slowed process at time $O(n\log n/\beta)$ if all of the following events occur:
  \begin{description}
    \item[Event $A_1$] The time to accumulate $I_{n-1}$ frogs at $v_{n-1}$ in the first
      step of the process is at most $O(Cn/\beta)$.
    \item[Event $A_2$] For all $J+1 \leq k \leq n-1$, it holds that
      $I_k\leq d^k$.
    \item[Event $A_3$] It holds that $I_{n-1}+\cdots+I_{J+1}\leq n\log n$.
  \end{description}
  Indeed, suppose all these events occur.
  For $J+2\leq k\leq n-1$, let
  \begin{align*}
    T_{k}=\max\biggl(3,\; \ceil[\bigg]{\frac{10000I_{k-1}}{\beta}},\;\ceil[\bigg]{\frac{I_k}{\beta}} \biggr).
  \end{align*}
  
  From event $A_1$, there will be $I_{n-1}$ frogs at $v_{n-1}$
  by time $O(Cn/\beta)$ or sooner, starting the stage of the slowed process in which it
  evolves according to $\FM(v_{n-1},\proc)$.
  By \thref{def:I},
  the process $\FM(v_{n-1},I_{n-1})$ sends at least $\beta t/10000$ frogs to
  $v_{n-2}$ in $t$ steps for all $\max(3,\ceil{I_{n-1}/\beta})\leq t\leq d^{n-1}$.
  From event $A_2$, we have $I_{n-1}\leq d^{n-1}$ and $I_{n-2}\leq d^{n-2}$.
  Hence $T_{n-1}$ lies between $\max(3,\ceil{I_{n-1}/\beta})$ and $d^{n-1}$,
  and therefore $\FM(v_{k-1},I_{k-1})$ sends at least $I_{n-2}$ frogs to $v_{n-2}$
  in $T_{n-1}$ steps.  
  By our construction of the slowed process, this kicks off the next stage of the process,
  which is coupled to $\FM(v_{n-2},\proc)$.
  By identical reasoning, $\FM(v_{n-2},I_{n-2})$ sends at least $I_{n-3}$ frogs to $v_{n-3}$
  in $T_{n-2}$ steps. Continuing in this way, we send at least $I_{n-4}$ frogs to $v_{n-4}$
  in another $T_{n-3}$ steps, and so on, culminating with the arrival of $I_{J+1}$ frogs 
  to $v_{j+1}$.
  Finally, let
  \begin{align*}
    T_{J+1} &= \max\biggl(3,\; \ceil[\bigg]{\frac{10^5 n\log n}{\beta}},\;\ceil[\bigg]{\frac{I_{J+1}}{\beta}} \biggr).
  \end{align*}
  By the definition of $J$, we have $10^5 n\log n/\beta\leq d^{J+1}$. From $A_2$, we have 
  $I_{J+1}\leq d^{J+1}$. Thus $T_{J+1}$ lies between $\max(3,\ceil{I_{J+1}/\beta})$
  and $d^{J+1}$, from which it follows that $\FM(v_{J+1},I_{J+1})$ sends at least $10n\log n$
  frogs to $v_J$ in $T_{J+1}$ steps. All together, 
  we send at least $10 n\log n$ frogs to $v_J$ in time
  \begin{align} \label{eq:10nlogn.time}
    O\biggl(\frac{Cn}{\beta}\biggr) + T_{n-1}+\cdots + T_{J+1}.
  \end{align}
  Assuming event $A_3$ holds, we have $T_{n-1}+\cdots+T_{J}=O(n\log n/\beta)$.
  Thus, \eqref{eq:10nlogn.time} is
  $O(n\log n/\beta)$ for large enough $n$ (depending on $C$), completing the proof of the claim.
  
  All that remains is to show that $A_1\cap A_2\cap A_3$ occurs with
  probability at least $1-e^{-Cdn}$.
  Let $c$ be a constant to be chosen later (it will depend only on $C$).
  To bound the probability of $A_1$, observe that $cdn$ frogs are frozen at $v_{n-1}$ in the first
  stage of the process by time $O(cn/\beta)$ with probability at least $1- e^{-cdn}$ by
  \thref{lem:initial.frogs}.
  By \thref{prop:I_k.tail}, we have $\P[I_{n-1}\leq  cdn]\geq 1 - O(1)e^{-\Omega(cdn)}$.
  These two facts together show that $A_1^c$ occurs with probability $O(1)e^{-\Omega(cdn)}$.\,
  provided the implicit constant in big-O expression in the definition of $A_1$ is chosen
  large enough.
  
  To bound the probability of $A_2$, first observe that $d^J\geq 10^5 n\log n/\beta$. Then apply
  \thref{prop:I_k.tail} and obtain the inequality
  \begin{align*}
    \P[I_k > d^k] &\leq \P[I_k > d^{k-J}(10)^5n\log n /\beta]\\
      &\leq O(1)\exp\Bigl(-\Omega\bigl(d^{k-J}n\log n/\beta\bigr)\Bigr)
  \end{align*}
  for all $J+1\leq k\leq n-1$. Hence, by a union bound,
  \begin{align*}
    \P[A_2^c] &\leq\sum_{k=J+1}^{n-1}O(1)\exp\Bigl(-\Omega\bigl(d^{k-J}n\log n/\beta\bigr)\Bigr)
      = O(1)e^{-\Omega(dn\log n/\beta)}.
  \end{align*}
  For large enough $n$ (depending only on $c$ and $\beta$), this is bounded
  by $e^{-c dn}$.

  Last, we consider the event $A_3$. Let $\oI_k=\min(I_k,d^k)$, so that $\oI_k$ has
  an exponential tail by \thref{prop:I_k.tail}.
  By \thref{prop:exp.sum.bound},
  \begin{align*}
    \P\bigl[\oI_{n-1} + \cdots + \oI_{J+1} > n\log n\bigr] &\leq e^{-c dn}
  \end{align*}
  once $n$ is large enough relative to $d$ and $c$.
  If $A_2$ holds and $\oI_{n-1} + \cdots + \oI_{J+1} \leq n\log n$, then $A_3$ holds as well,
  showing that $\P[A_3^c]\leq 2e^{-cdn}$ for large enough $n$, depending on $c$ and $\beta$.
  We now have
  \begin{align*}
    \P[A_1^c] + \P[A_2^c] + \P[A_3^c] = O(1)e^{-\Omega(cdn)}.
  \end{align*}
  The proof is now completed by choosing $c$ large enough that this is smaller than $e^{-Cdn}$.
\end{proof}

\begin{proof}[Proof of \thref{prop:TJ}]
  Set $b=4\log d$ and apply \thref{cor:time.change} to the result of
  \thref{prop:TJ.nonbacktracking}.
\end{proof}

\subsection{Establishing \thref{prop:cleanup}}

In this section we prove \thref{prop:cleanup}. 
Note that this is a result about random walks on trees, not the frog model.
It will be based on the following random walk estimate.
Recall from \eqref{eq:Jdef} that
$J=\floor{ \log_d(10^5n\log n/\beta)}$, where $\mu=(3+\beta)d(d+1)$.
\begin{prop} \thlabel{lem:hitting_prob}
Consider a single random walk
on $\mathbb T_d^n$ started at $v_J$ and assume that $n\geq n_0$ for some sufficiently large absolute
constant $n_0$. The walk visits $v_0$ in less than $4(10)^5 n \log n/\beta$ steps with probability at least $1/ 3 \log_d n$.
\end{prop}
Using this, the proof of \thref{prop:cleanup} is easy:

\begin{proof}[Proof of \thref{prop:cleanup}]
  By \thref{lem:hitting_prob}
the probability that none of the $10n\log n$ frogs at $v_J$ reaches $v_0$ in
$4(10)^5 n \log n/\beta$ steps is at most
\begin{align*}
  \biggl(1 -\frac{1}{3 \log_d n}\biggr)^{10 n \log n} &\leq e^{-3n\log d}.\qedhere
\end{align*}
\end{proof}

Now we devote the rest of this section to establishing the random walk estimate.  
Its proof works by decomposing the random walk as a simple random walk on the spine $\{v_0,\ldots,v_n\}$
with excursions off of it. We start with a preliminary lemma to compute the expected length of the excursions.

\begin{lemma} \thlabel{lem:E_tau}
  Let $\tau_k$ be the number of steps to hit either $v_{k-1}$ or $v_{k+1}$
  for a simple random walk on $\TT^n_d$ starting at $v_k$. Then $\E \tau_k=d^{k-1}(d-1)/2$.
\end{lemma}

\begin{proof}
  The time to hit $v_{k-1}$ or $v_{k+1}$ is the same as the time 
  by random walk starting at $v_k$ on the weighted graph shown in Figure~\ref{fig:network}
  to hit the leftmost vertex, $\{v_{k-1},v_{k+1}\}$.
  The random walk moves at each step to a neighbor chosen with probability proportionate to the
  weight of the edge. The graph has been obtained from $\TT^n_d$ by identifying $v_{k-1}$ and $v_{k+1}$,
  identifying all children of $v_k$ other than $v_{k-1}$, and identifying all distance~$k$ descendants
  of $v_k$ for each $k\geq 2$.
  
  \begin{figure}
    \begin{tikzpicture}[scale=2,tv/.style={circle,fill,inner sep=0,
    minimum size=0.15cm,draw}]
      \path (0,0) node[tv,label={above:$\{v_{k-1},v_{k+1}\}$}] (leftmost) {}
            (1,0) node[tv,label={above:$v_k$}] (k) {}
            (2,0) node[tv] (k-1) {}
            (3,0) node[tv] (k-2) {}
            (4.5,0) node[tv] (n1) {}
            (5.5,0) node[tv] (n0) {} ;
      \draw[thick] (leftmost)--node[auto,swap] {$2$} (k)--node[auto,swap] {$d-1$} 
        (k-1)--node[auto,swap] {$d(d-1)$} (k-2)   
        (n1) --node[auto,swap] {$d^{k-1}(d-1)$} (n0)
        (k-2)-- +(.5,0)
        (n1)--  +(-.5,0);
      \path (k-2)--node {$\cdots$} (n1);
    \end{tikzpicture}
    \caption{The following collections of vertices from $\TT^n_d$ have been identified 
    in this graph: $v_{k-1}$ and $v_{k+1}$; all children of $v_k$ other than $v_{k-1}$;
    and the distance~$k$ descendants of $v_k$ for each $k\geq 2$. Random walk moving
    with probability proportionate to the edge weights starting at $v_k$ and stopping
    at $\{v_{k-1},v_{k+1}\}$ is the same as random walk on the original graph, viewing vertices
    as blocked together.}\label{fig:network}
  \end{figure}
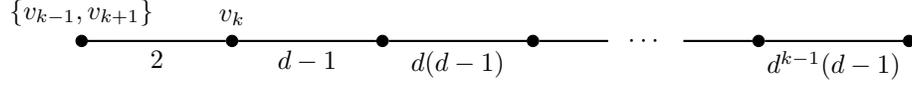

  The expected hitting time is easily computed using electrical network theory.
  By \cite[Proposition~2.20]{LP}, which is a result first obtained in \cite{Tetali}, the hitting
  time has expectation $\sum_x \pi(x)v(x)$, where the sum is over all vertices in the graph,
  $\pi(x)$ denotes the sum of the weights of edges incident to $x$, and $v$ is the voltage that assigns
  0 to the vertex $\{v_{k-1},v_{k+1}\}$ and that creates a unit current flow from $v_k$
  to $\{v_{k-1},v_{k+1}\}$. This voltage assigns $1/2$ to all vertices other than $\{v_{k-1},v_{k+1}\}$.
  The expected hitting time is then
  \[
    \frac12 \sum_{i=0}^{k-1} \Bigl(d^i(d-1) + d^{i+1}(d-1)\Bigr)+  \frac12(d+1)=\frac12 d^{k-1}(d-1).\qedhere
  \]
\end{proof}

Next, we compute the expected number of visits to each vertex along the spine before $v_0$ is hit.
We condition on the walk hitting $v_0$ before $v_{J+1}$, as this will simplify our eventual proof.
\begin{lemma} \thlabel{lem:E_Vk}
  Consider a random walk on $\TT^n_d$ starting at $v_J$. Let $\sigma_k$ be the first time
  that the walk hits $v_k$. Let $V_k$ be the total number of visits to $v_k$ up to
  time $\sigma_0$. For $1\leq k\leq J$,
  \begin{align*}
    \E [V_k \mid \sigma_0<\sigma_{J+1} ] = 2k \Bigl(1- \f{k}{ J+1} \Bigr).
  \end{align*}
\end{lemma}
\begin{proof}
  Let $E=\{\sigma_0<\sigma_{J+1}\}$.
  Recall that $\P[\sigma_k<\sigma_{J+1}]= 1 / (J+1-k)$.
	Conditioned on $E$, the walk will visit $v_k$ at least once for all
  $k \leq J$. The number of returns to $v_k$ after first visiting it
  is a geometric random variable with parameter $1-\P[V_k = 1 \mid E]$. We can then write
\begin{align}\E[V_k \mid E]=\frac{1}{\P[V_k=1 \mid \sigma_0 < \sigma_{J+1}]} 
&= \f{ \P[ \sigma_0 < \sigma_{J+1} ] }{\P[ \sigma _0 < \sigma_{J+1} \text{ and } V_k = 1]}\nonumber\\
 &= \frac{1}{(J+1)\P[ \sigma _0 < \sigma_{J+1} \text{ and } V_k = 1]}.\label{eq:Evk}
\end{align}
We claim that 
\begin{align}
  \P[\sigma_0 < \sigma_{J+1} \text{ and } V_k = 1]  = \frac{1}{2(J+1-k)k}.\label{eq:bot}
\end{align}
This is because to reach $v_1$ before $v_{J+1}$, the walk necessarily visits $v_k$ before $v_{J+1}$, which
occurs with probability $1/(J+1-k)$. To visit $v_k$ only once, on arriving at $v_k$ it must immediately move to $v_{k-1}$, which occurs with probability $1/2$. Then it must reach $v_0$ before $v_k$, which occurs with probability $1/k$. Combining \eqref{eq:Evk} and \eqref{eq:bot} gives the claimed formula.
\end{proof}

\begin{proof}[Proof of \thref{lem:hitting_prob}]
  Let $(S_t)$ be a simple random walk on $\TT_d^n$ starting at $v_k$.
  Define $\Ss=\{v_0,\ldots,v_n\}$. The Markov property of random walk shows that the restriction
  to $\Ss$ of the path of $(S_t)$ is distributed as the path of a simple random walk on $\Ss$.
  Let $\sigma_k=\inf\{t\colon S_t=v_k\}$, the hitting time of $v_k$, as in
  \thref{lem:E_Vk}. 
  Set 
  \begin{align*}
    F &= \{\sigma_0 \leq 4(10)^5 n \log n/\beta\},\\
    E &= \{\sigma_0 < \sigma_{J+1}\}.
  \end{align*}
  Our goal is to bound $\P[F]$ from below. 
  A simple estimate gives
\begin{align}\P[F] \geq \P[F \cap E] = \P[F \mid E] \P[E] = \f{ \P[F \mid E] } {J}.\label{eq:bayes}\end{align}
In light of \eqref{eq:bayes} it suffices to prove that $\P[F \mid E] \geq 1/2$.

Let $V_k =| \{t \leq \sigma_0 \colon S_t = v_k\}|$, the total number of visits to $v_k$ before the walk
hits $v_0$. Let $\tau_k(i)$ be the number of steps it takes the walk to reach $v_{k-1}$ or $v_{k+1}$
starting from the $i$th time the walk arrives at $v_k$. We then decompose $\sigma_0$ as
\begin{align*}
  \sigma_0 &= \sum_{k=1}^n\sum_{j=1}^{V_k} \tau_k(j).
\end{align*}
Conditional on $E$, the random variables $V_k$ and $\tau_k(j)$ are mutually independent for all $j$
and $k$. By Wald's lemma,
\begin{align}\label{eq:sigma}
  \E[\sigma_0\mid E] &= \sum_{k=1}^J \E[V_k\mid E]\,\E[\tau_k(j)\mid E].
\end{align}
We need only consider $J$ summands in \eqref{eq:sigma}, since conditional on $E$ the walk does not
move beyond $v_J$ before hitting $v_0$.
For all $j$, the random variable $\tau_k(j)$ is independent of $E$ and is distributed as
$\tau_k$ from \thref{lem:E_tau}. Therefore, by \thref{lem:E_tau,lem:E_Vk},
\begin{align*}
  \E[\sigma_0\mid E] &= \sum_{k=1}^J k\Bigl(1-\frac{k}{J+1}\Bigr)d^{k-1}(d-1).
\end{align*}
We claim that this is $O(d^J)$. Indeed, using the bound $k\bigl(1-k/(J+1)\bigr)\leq J+1-k$
and making the substitution $j=J+1-k$ in the second line,
\begin{align*}
  \E[\sigma_0\mid E] &\leq (d-1)\sum_{k=1}^J (J+1-k)d^{k-1}\\
    &= d^{J-1}(d-1)\sum_{j=1}^J jd^{1-j} \\
    &\leq d^{J-1}(d-1) \sum_{j=1}^{\infty} j d^{1-j} = d^{J-1}(d-1)\bigl(1-d^{-1}\bigr)^{-2}= \frac{d^{J+1}}{d-1}.
\end{align*}

Notice that $d^{J+1}\leq 10^5d n\log n/\beta$ and apply Markov's inequality to obtain
\begin{align*}
  \P[\sigma_0 > 4(10)^5n\log n/\beta\mid E]\leq \frac{d}{4(d-1)}\leq\frac12.
\end{align*}
Applying this to \eqref{eq:bayes} gives
\begin{align*}
  \P[\sigma_0\leq 16(10)^5n\log n/\beta] &\geq\frac{1}{2J}\geq\frac{1}{3\log_d n},
\end{align*}
with the last inequality holding for all sufficiently large $n$, with no dependence on $d$.
\end{proof}

\section{Slow cover time for small $\mu$}

\label{sec:lower}

We now give our lower bound on the cover time for small enough $\mu$. 

 \begin{thm} \thlabel{thm:lb}
  Let $\Cc$ be the cover time for the frog model on $\TT_d^n$ with initial frog counts given 
  by an independent collection of random variables 
  $(\eta(v))_{v\in\TT_d^n\setminus\{\emptyset\}}$, where $\E\eta(v)\leq\mu$ for all vertices $v$.
  Suppose that $\mu\leq\min(d^{1-\epsilon},d/100)$ for any $0 < \epsilon <1$.
  For some absolute constant $c>0$, 
  \begin{align*}
    \P\bigl[ \Cc < e^{c\sqrt{\epsilon n\log d}} \bigr] &\leq e^{-c\sqrt{\epsilon n\log d}}
  \end{align*}
  for $n\geq \log d/c^2\epsilon$.
 \end{thm}

This bound is effective even for $\mu$ as large
as $d/100$:
\begin{cor}\thlabel{cor:lb.linear}
  Let $\Cc$ be the cover time as above, assuming only that $\mu\leq d/100$. Then
  for some absolute constant $c>0$,
  \begin{align*}
    \P\bigl[\Cc < e^{c\sqrt{n}}\bigr] &\leq e^{-c\sqrt{n}}
  \end{align*}
  for $n\geq ((\log d)/c)^2$.
\end{cor}
\begin{proof}
  Apply \thref{thm:lb} with $\epsilon=\log_d100$.
\end{proof}

We extend the usual notion of the distribution $\Ber(\mu)$ to $\mu>1$ by setting it to be 
the unique distribution on 
$\{ \floor{\mu} , \ceil{\mu}\}$ with mean $\mu$.
For most of this section, we consider the frog model with
i.i.d.-$\Ber(\mu)$ initial conditions. 
We then apply \thref{cor:comparison} to allow for more general initial conditions.
 
The proof hinges on the following result that we will prove inductively.
Define $\mathbb T_d^{H*}$ to be the $d$-ary tree of height $H$ with an extra vertex, $y$, 
attached to the root.

\begin{prop}\thlabel{thm:induction} 
  For some absolute constant $C>0$, the following statement holds for all $d\geq 2$
  and $\mu\leq d/100$. 
  Consider the frog model on $\TT_d^{H*}$ with one initially active frog at the root, none at $y$,
  and i.i.d.-$\Ber(\mu)$ sleeping frogs at the remaining vertices, and with frogs frozen on moving
  to $y$.
  Let $X^{(j,H)}$ be the number of frogs frozen at $y$ by time $2^j$.
  Define $H_j = H_j(d,\mu)$ by
  \begin{align*}
    H_1 &=1,\\
    H_j &= \ceil[\Bigg]{\frac{Cj\bigl( \log(1+\mu) + j \bigr)}{\log\bigl( \frac{d}{1+\mu} \bigr) }},
          \qquad j\geq 2.
  \end{align*}
  For any $j\geq 1$, if $n\geq H_j$, then
  \begin{align}\label{eq:induction2}
    \E X^{(j,n)} &\leq \frac{.8}{1 + \frac{2d}{d-1}\mu}.
  \end{align}
\end{prop}

Most of this section is devoted to proving \thref{thm:induction}.
Before we turn to this, we prove \thref{thm:lb} from it.
First, it is a small task to remove the freezing of frogs from \thref{thm:induction},
showing that the expected number
of returns to the root within time $2^j$ in our usual frog model on $\TT^{H_j}_d$ is $O(1)$. 
\begin{cor} \thlabel{cor:unfrozen}
  Suppose that $\mu\leq d/100$, and let
  \begin{align*}
    j=j(d,n,\mu) = \max\{i\colon H_i\leq n-1\},
  \end{align*}
  where $H_i$ is the sequence defined in \thref{thm:induction}.
	Let $R$ be the total number of visits to the root of $\mathbb T_d^n$ within
  time $2^{j}$ in the frog model with 
  initial frog counts given by $\eta(v)$
  for $v\in\TT_d^n\setminus\{\emptyset\}$. If $\E\eta(v)\leq\mu$ for all vertices $v$, then
  $\E R \leq 4$.
\end{cor}

\begin{proof}
It suffices to prove this result under i.i.d.-$\Ber(\mu)$ initial conditions, by \thref{cor:comparison}
and the maximality of $\Ber(\mu)$ in the pgf order mentioned in Appendix~\ref{sec:comparison}.
Now, consider the following modification of the frog model.
Let the initial frog take a step. Next, run the frog model for $2^j$
steps with frogs frozen at the root, and kill all frogs that were woken but did not reach the
root. Let $R_1$ be the number of frogs frozen at the root. Now, let each
of these frogs take one more step, and then run the frog model for another $2^j$ steps with frogs
frozen at the root, and then again kill any frogs that were woken but did not reach the root.
Let $R_2$ be the number of frogs frozen at the root after this stage.
Continue in this way to define $R_i$ for $i\geq 3$.
As every frog is allowed to run for at least $2^j$ steps
before being killed, every visit to the root in the usual frog model in the first $2^j$
steps occurs eventually in this modified process. Hence, $R\leq \sum_{i=1}^{\infty}R_i$.

Defining $R_0=1$, we claim that
\begin{align}
  \E R_{i+1} &\leq (1+\mu)\E X^{(j,n-1)}\E R_i\label{eq:next.term}
\end{align}
for all $i\geq 0$.
We prove this statement now. After the $i$th step of the process, there are $R_i$ frogs
at the root. Let $N_i$ be the number of active frogs at level~$1$ after they take their next steps.
Fix $1\leq k\leq N_i$, and suppose that the $k$th of these frogs follows the path $(S_0,S_1,\ldots)$ 
from this point on. Consider the original (unmodified) frog model with the following changes:
\begin{enumerate}[(i)]
  \item Add an initially active frog with path $(S_0,S_1,\ldots)$;
  \item delete all other frogs at vertex~$S_0$, and delete the frog at the root;
  \item freeze frogs on moving to the root.
\end{enumerate}
Let $X_k$ be the number of frogs frozen at the root after $2^j$ steps in this modified process.
By a subadditivity
property of the frog model, $R_{i+1}\leq\sum_{k=1}^{N_i} X_k$. Now, we think of the root vertex
as $y$, and we think of $X_k$ as counting
the number of visits to $y$ in a frog model on $\TT^{(n-1)*}_d$ with frogs frozen at $y$, except that
because of killing frogs,
some vertices of $\TT^{(n-1)*}_d$ have no sleeping frogs on them. Thus, conditional on $N_i$,
we have $X_k\preceq X^{(j,n-1)}$. Hence,
\begin{align*}
  \E[R_{i+1}\mid N_i] &\leq (\E X^{(j,n-1)})N_i.
\end{align*}
Taking expectations and observing that $\E[N_i\mid R_i]\leq (1+\mu)R_i$
completes the proof of \eqref{eq:next.term}.

By \thref{thm:induction} and our choice of $j$,
\begin{align*}
  (1+\mu)(\E X^{(j,n-1)}) &\leq \frac{.8(1+\mu)}{1+\frac{2d}{d-1}\mu}\leq .8.
\end{align*}
It now follows from \eqref{eq:next.term} that
\begin{align*}
  \E R &\leq \sum_{i=1}^{\infty}\E R_i\leq \sum_{i=1}^{\infty}(.8)^i = 4. \qedhere
\end{align*}
\end{proof}

\thref{cor:unfrozen} shows that in the frog model on $\TT^n_d$, there are few visits to the
root by time $2^{j(d,n,\mu)}$. To bound the cover time, we observe that once
all frogs are active, many visits to the root will occur. We first give a random walk estimate.

\begin{lemma}\thlabel{lem:rw.root}
  For some absolute constants $a,b>0$ the following statement holds.
  Suppose that $n\log d/a\leq t\leq d^n$.
  Then a random walk on $\TT^n_d$ with arbitrary starting position has probability
  at least $btd^{-n}$ of hitting the root in its first $t$ steps.
\end{lemma}
\begin{proof}
  One could prove more precise estimates in the same way as \thref{lem:root.miss}.
Since we do not need any precise formula, we take a simpler approach. We can assume the walk
  starts at a leaf, as this is the worst-case scenario. Now, partition the walk into excursions
  away from level~$n$. The length of each excursion has an exponential tail, since the probability
  that a random walk on $\ZZ$ from $0$ with a bias to the right is negative after $k$ steps decays
  exponentially in $k$. By \thref{prop:exp.sum.bound}, the probability of having
  $\epsilon t$ or fewer excursions from level~$n$ in time $t$ is at most $e^{-ct}$ for 
  absolute constants $\epsilon$ and $c$. On each excursion, the walk
  has probability $(d-1)/(d^n-1)\geq d^{-n}$ of visiting the root.
  Thus, in $\ceil{\epsilon t}$ excursions, the probability that the root will not be visited
  is at most
  \begin{align*}
    (1 - d^{-n})^{\ceil{\epsilon t}}\leq e^{-\epsilon t d^{-n}}.
  \end{align*}
  Combining these two estimates, the root is visited in time~$t$ with probability at least
  \begin{align*}
    1 - e^{-\epsilon td^{-n}} - e^{-ct}.
  \end{align*}
  Since $t\leq d^n$, we can apply the inequality $1-e^{-x}\geq x/2$, which holds for $x\in[0,1]$,
  to get
  \begin{align*}
    1 - e^{-\epsilon td^{-n}} - e^{-ct}\geq \frac{\epsilon t d^{-n}}{2} - e^{-ct}.
  \end{align*}
  Choosing $a$ small enough, this is $\Omega(td^{-n})$.
\end{proof}

\begin{proof}[Proof of \thref{thm:lb}]
  Define $j=j(d,n,\mu)$ as in \thref{cor:unfrozen}.
  We start by estimating $j$.
  Directly calculating from the definition of $H_i$ in \thref{thm:induction}, we find that
  if $i\geq\log d$, then
  \begin{align*}
    H_i &\leq \frac{Ci^2}{\epsilon\log d}
  \end{align*}
  for some absolute constant $C$.
  If we set $i=\ceil{c\sqrt{\epsilon n\log d}}$ for $c=(2C)^{-1/2}$
  and assume $n\geq c^{-2}\log d/\epsilon$ so that $i\geq\log d$,
  then we have $H_i\leq n-1$. Hence $j\geq c\sqrt{\epsilon n\log d}$.
  It is also straightforward to see that $j=O(\sqrt{n\log d})$.

  It does us no harm to assume that $\mu>.01$.
  For technical reasons, we will also assume that the expected number of sleeping frogs 
  is exactly $\mu$ at each site,
  rather than just being bounded by $\mu$. To see that it suffices to prove the theorem under
  this extra assumption, for each site with expected count strictly smaller than $\mu$, independently
  add a random number of extra frogs (distributed arbitrarily) to bring the mean up to $\mu$,
  and observe that this can only decrease the cover time.

  Define the event $A = \{\Cc < 2^{j-1} \}$. 
  We will prove that $\P[A]\leq C 2^{-j}$ for some absolute constant $C$.
  By the lower bounds on $j$, this proves the theorem with an extra constant $C$ in
  front of the bound, which we can eliminate by decreasing $c_\epsilon$ or $c$ slightly.
  Conditional on $A$, all frogs are awake at time $2^{j-1}$, and they move from this time on as independent
  simple random walks.  
  We can apply \thref{lem:rw.root} with $t=2^{j-1}$, since $n\log d/a\leq 2^{j-1}\leq d^n$ for large
  enough $n$, showing that each walk hits the root by time $2^j$ with probability at least
  $b2^{j-1}d^{-n}$. 
  Let $R$ be the total number of visits to the root by time $2^j$ and let $U$ be the total number
  of frogs in the system.
  Bounding $R$ from below by counting the visits to the root only for times
  in $[2^{j-1},2^j]$, we obtain
  \begin{align*}
    \E[R\mid A] &\geq b2^{j-1}d^{-n}\E[U\mid A].
  \end{align*}
  By a simple coupling, the event~$A$ is more likely the larger $U$ is.
  That is, the random variables $\1_A$ and $U$ are positively associated,
  from which it follows that
  $\E[U\mid A]\geq \E U\geq \mu d^n$, recalling that we have assumed that each site
  has exactly mean $\mu$ sleeping frogs. Thus, $\E[R\mid A] \geq b2^{j-1}\mu$.
  But by \thref{cor:unfrozen}, we have $\E R \leq 4$. Rearranging the simple bound $\E R \geq \E[ R  \mid A ] \P[A]$  gives 
  \begin{align*}
    \P[A] &\leq \frac{4}{b2^{j-1}\mu} = O\bigl(2^{-j}\bigr),
  \end{align*}
  under our assumption that $\mu\geq .01$.
\end{proof}

\subsection{Tagging frogs}
The remainder of Section~\ref{sec:lower} is devoted to proving \thref{thm:induction}.
Fix integers $H,h,j\geq 1$, and consider the frog model on $\TT_d^{(H+h)*}$ with frogs frozen 
on moving to $y$, starting with one frog at the root, and with i.i.d.-$\Ber(\mu)$
frogs per site at all vertices besides the root and $y$. 
Let $\Ll_i$ denote the set of vertices at level~$i$ of $\TT_d^{(H+h)*}$, 
taking $0$ as the level of the root and $-1$ as the level of $y$.

Our plan is to advance
the induction in \thref{thm:induction} by supposing that $X^{(j,H)}$ satisfies the inductive
hypothesis and then showing that $X^{(j+1,H+h)}$ does as well, for a good choice of $h$.
The idea of the proof is to assign each frog
a tag that changes at various times in the process. When one frog wakes another,
the newly woken frog starts with the same tag as its waker. If a frog
is woken by two frogs with different tags arriving simultaneously, choose
any procedure to decide between the frogs; this detail will prove irrelevant.
In the following set of rules, when a frog changes its tag on arriving at a given
vertex, the newly woken frogs inherit the new tag, not the old one.
\begin{itemize}
  \item The initially active frog at the root has tag~$A$.
  \item If an $A$-tagged frog reaches $\Ll_h$, its tag changes to $B_0$.
  \item If a $B_i$-tagged frog moves from $\Ll_{h-1}$ to $\Ll_h$, its tag changes to $B_{i+1}$.
  \item If a $B_i$-tagged frog moves from $\Ll_h$ to $\Ll_{h-1}$ at time $2^j$ or after,
    its tag changes to $C_0$.
  \item If a $C_i$-tagged frog moves from $\Ll_{h-1}$ to $\Ll_h$, its tag changes to $C_{i+1}$.
  \item At time $2^{j+1}+1$, all frogs are stripped of their tags.    
\end{itemize}
Note that frogs are retagged every time they move forward in the tree to $\Ll_h$.
The only other time a frog receives a new tag is when a $B_i$-tagged frog
moves backward from $\Ll_h$ to $\Ll_{h-1}$ at time $2^j$ or later, in which case
its tag changes to $C_0$.

We will use three different estimates to bound the number of tagged frogs.
When frogs with any tag are between the root and $\Ll_h$, we dominate them by branching random
walks using the estimates in \thref{lem:kbrw1,lem:kbrw2}.
When a frog moves forward in the tree to a vertex $v\in\Ll_h$ and is given tag~$B_i$, we estimate
the number of $B_i$-tagged frogs emerging from $v$ back to $\Ll_{h-1}$
using \eqref{eq:induction2}, the inductive hypothesis.
We have very little control over the number of $C_0$-tagged particles emerging from $v$ back
to $\Ll_{h-1}$. Here, we use \thref{lem:all_awake}, which we call the \emph{all-awake bound}
since it simply assumes that all frogs in the subtree rooted at $v$ are initially awake.
The key to the argument is that we retain control over the number of $C_{i+1}$-tagged frogs:
Whenever a $C_i$-tagged frog moves forward to a vertex $v\in\Ll_h$ and is retagged as $C_{i+1}$,
it does so after time $2^j$. Since we only care about the process up to time $2^{j+1}$,
we can control the number of $C_{i+1}$-tagged frogs emerging from $v$ back to $\Ll_{h-1}$
using the inductive hypothesis rather than the all-awake bound. Thus, although
there will be many $C_0$-tagged frogs, the number of $C_i$-tagged frogs for $i\geq 1$
will not spiral out of control.

\subsection{Branching random walk and all-awake bounds}

As mentioned above, we control the frog model by dominating it by
branching random walk and by simply assuming that all frogs in a given subtree are initially
awake. We start with this second bound.

\begin{lemma}[All-awake bound]\thlabel{lem:all_awake}
  Consider $\mathbb T_d^{H*}$ for arbitrary $H\geq 1$ with one particle at the root, none at $y$, and 
  i.i.d.-$\Ber(\mu)$ particles at the remaining vertices. Let all particles perform discrete-time
  random walks frozen at $y$. Let $W$ be the total number of particles frozen at $y$ after $t$ time steps. 
  For some constant $c_1$,
  \begin{align*}
    \E W\leq c_1\mu t.
  \end{align*}
\end{lemma}

\begin{proof}
 For any $0\leq k\leq H-1$, a particle initially at level~$k$ of the tree visits $y$ before
 the leaves with probability no more than $d^{-k-1}$. Initially, there is one particle 
 at level~$0$ and an expected $\mu d^k$ particles at level~$k$ for each $1\leq k\leq H-1$.
 Only particles starting at level~$t-1$ or less can reach $y$ in time~$t$.
 Hence, the expected number of particles that
 reach the root in $t$ steps without ever being at a leaf is at most
 \begin{align}\label{eq:no.leaf.visits}
   d^{-1} + \sum_{k=1}^{\min(H,t)-1} \mu d^k d^{-k-1} \leq  \frac{1+t\mu}{d}.
 \end{align}
 
 Now, consider consider a particle at a leaf. It has probability no more than $d^{-H}$ of visiting
 $y$ before revisiting level~$H$. In time~$t$, it makes no more than $t$ of these excursions from
 the leaves. Thus, the probability that a given particle at a leaf visits $y$ in its next $t$ steps is at most
 \begin{align*}
   1-(1-d^{-H})^t \leq 1 - e^{-2d^{-H}t}\leq 2td^{-H}.
 \end{align*}
 The first inequality above uses the bound $1-x\geq e^{-2x}$, which holds for all $x\in[0,1/2]$.
 The total expected number of particles in the tree is $1+ \mu(d+\cdots+d^H)$.
 The expected number of particles that visit $y$ before time~$t$,
 starting from a leaf or after visiting a leaf, is therefore at most
 \begin{align}\label{eq:leaf.visits}
   2td^{-H}\bigl(1+ \mu(d+\cdots+d^H)\bigr) \leq 2t\biggl(d^{-H} + \frac{\mu}{1-d^{-1}}\biggr).
 \end{align}
 Combining \eqref{eq:no.leaf.visits} and \eqref{eq:leaf.visits},
 \begin{align*}
   \E W &\leq \frac1d + \biggl(\frac{\mu}{d} + 2d^{-H} + \frac{2\mu}{1-d^{-1}}\biggr)t\\
     &\leq \frac12 + \biggl(\frac{\mu}{2} + 1 + 4\mu\biggr)t = O(\mu t).\qedhere
 \end{align*}
\end{proof}

Next, we prove several bounds whose proofs are essentially comparisons of the frog model to branching
random walk.
The first step is to describe a supermartingale $w_\theta(\xi_t)$
given as a function of the frog model.

\begin{lemma}\thlabel{lem:martingale}
  Consider $\mathbb T_d^{h*}$ with a single active frog at a specified vertex $v_0$, 
  no frogs at the ancestors of $v_0$ (including $y$), and i.i.d.-$\Ber(\mu)$ sleeping frogs at 
  the other vertices. Run the frog model with frogs frozen on arrival to $y$ and to $\Ll_h$.
  (When a frog arrives at $\Ll_h$, we consider the frogs there woken but immediately frozen.)
  Let $\mathscr{F}_t$ be the $\sigma$-algebra representing the information revealed after 
  $t$ steps of this process.
  Let $\xi_t$ be a point process on $\TT_d^{h*}$ made up of the locations of each woken frog
  after $t$ steps.
  For any $v\in\xi_t$, let $L(v)$ denote the level of $v$ in the tree, and define
  \begin{align*}
    w_\theta(\xi_t) &= \sum_{v\in\xi_t}\theta^{- L(v)}.
  \end{align*}
  If $\mu\leq (d-1)^2/4d$, then
  there exist positive real numbers $\theta_0$ and $\theta_1$ satisfying
  \begin{align}
    \theta_0 &\leq 1 + \frac{2d}{d-1}\mu,\label{eq:theta_0}\\
    \theta_1 &\geq d-\frac{2d}{d-1}\mu. \label{eq:theta_1}
  \end{align}
  such that $w_{\theta_0}(\xi_t)$ and $w_{\theta_1}(\xi_t)$ 
  are supermartingales with respect to the filtration $\mathscr F_t$.
\end{lemma}
\begin{proof}
  Observe that
  \begin{align*}
    \E w_\theta(\xi_1) &= \biggl(\frac{1}{d+1}\theta + \frac{(1+\mu)d}{d+1}\theta^{-1}\biggr)w_\theta(\xi_0).
  \end{align*} 
  Solving a quadratic equation, we see that $\E w_\theta(\xi_1)=w_\theta(\xi_0)$ if 
  \begin{align*}
    \theta &= \frac{d+1 \pm \sqrt{(d+1)^2 - 4(1+\mu)d}}{2}.
  \end{align*}
  Let $\theta_0$ and $\theta_1$ be the smaller and larger of these solutions, respectively,
  which are positive real numbers if $0\leq\mu\leq (d-1)^2/4d$.
  Let $\mathscr{F}_t$ be the $\sigma$-algebra generated by the frog model up to time~$t$.
  Now, suppose that $\theta=\theta_0$ or $\theta=\theta_1$, and we will show that 
  $w_\theta(\xi_t)$ is a supermartingale.
  Consider a nonfrozen frog in $\xi_t$ at level~$i$.
  It jumps backward with probability~$1/(d+1)$,
  waking no frogs, and forward with probability $d/(d+1)$, possibly waking a $\Ber(\mu)$-distributed
  number of frogs. Thus, its expected contribution to $w_{\theta}(\xi_{t+1})$ is at most
  \begin{align*}
    \frac{1}{d+1}\theta^{-i+1} + \frac{(1+\mu)d}{d+1}\theta^{-i-1}=\theta^{-i}, 
  \end{align*}
  exactly its current contribution.
  The contribution to $w_{\theta}(\xi_{t+1})$ of each frozen frog in $\xi_t$ is the same as its
  contribution to $w_{\theta}(\xi_t)$,
  showing that
  \begin{align*}
    \E[ w_{\theta}(\xi_{t+1}) \mid \mathscr{F}_t] &\leq w_{\theta}(\xi_t).
  \end{align*}
  
  To prove \eqref{eq:theta_0} and \eqref{eq:theta_1}, observe that
  $\sqrt{(d+1)^2-4(1+\mu)d}$ is a concave function of $\mu$. It therefore lies above
  its secant line from $0$ to $(d-1)^2/4d$, yielding
  \begin{align*}
    \sqrt{(d+1)^2-4(1+\mu)d} &\geq d-1 - \frac{4d}{d-1}\mu.
  \end{align*}
  Applying this to the definitions of $\theta_0$ and $\theta_1$ gives the desired bounds.  
\end{proof}

\begin{lemma}[BRW bound, starting at root]\thlabel{lem:kbrw1}
    Consider $\mathbb T_d^{h*}$ with one initially active frog at the root, no frogs at $y$, 
    and i.i.d.-$\Ber(\mu)$ sleeping frogs at the other vertices. 
    Run the frog model with frogs frozen at $y$ and $\Ll_h$.
    Let $X$ and $N$ be the number of particles eventually frozen at $y$ and $\Ll_h$, respectively.
    (The random variable $N$ includes in its count the frogs that are
    woken at $\Ll_h$ and immediately frozen.)
    If $\mu\leq (d-1)^2/4d$, then
    \begin{align*}
      \E N \leq \Bigl(1+\tfrac{2d}{d-1}\mu\Bigr)^h, \qquad\text{and}\qquad
      \E X \leq \Bigl(d\bigl(1-\tfrac{2\mu}{d-1}\bigr)\Bigr)^{-1}.
    \end{align*}
\end{lemma}
\begin{proof}
  Let $T$ be the first time when all frogs are frozen.
  By \thref{lem:martingale}, the process $w_\theta(\xi_t)$ is a supermartingale
  for $\theta=\theta_0,\theta_1$.
  It is bounded at all times by 
  $\theta\ceil{\mu}\abs[\big]{\TT_d^{h*}}$,
  since the total number of frogs
  in the system is at most $\ceil{\mu}\abs[\big]{\TT_d^{h*}}$ and no frog goes below
  level $-1$.
  Hence, the optional
  stopping theorem applies and shows that $\E w_{\theta}(\xi_T)\leq 1$.
  The expected contribution to $w_{\theta_0}(\xi_T)$ by frogs frozen at $\Ll_h$ is
  \begin{align*}
    \theta_0^{-h}\E N\leq \E w_{\theta_0}(\xi_T)\leq 1.
  \end{align*}
  Then \eqref{eq:theta_0} gives the bound on $\E N$.
  Similarly, the expected contribution to $w_{\theta_1}(\xi_T)$ by frogs frozen at $y$ is
  \begin{align*}
    \theta_1\E X \leq \E w_{\theta_1}(\xi_T)\leq 1,
  \end{align*}
  and \eqref{eq:theta_1} gives us the bound on $\E X$.
\end{proof}

The previous lemma bounds the expected number of frogs at $\mathcal L_h$ and at $y$ when we have an initially active frog at the root. The next lemma makes similar bounds when the initially active frog is at $\mathcal L_{h-1}$.

\begin{lemma}[BRW bound, starting at level~$h-1$] \thlabel{lem:kbrw2}
Consider the frog model on $\TT_d^{h*}$  
with one initially active frog at some vertex $v_0\in\mathcal L_{h-1}$, no frogs at
ancestors of $v_0$ (including $y$), and i.i.d.-$\Ber(\mu)$ sleeping frogs elsewhere.
Run the frog model with frogs frozen at $y$ and $\Ll_h$.
Let $X$ and $N$ be the number of particles eventually frozen at $y$ and $\Ll_h$, respectively.
(Again, the frogs woken at $\Ll_h$ and immediately frozen are included in the count $N$.)
If $\mu\leq (d-1)^2/4d$, then
\begin{align*}
  \E N \leq 1 + \tfrac{2d}{d-1}\mu, \qquad\text{and}\qquad
  \E X \leq \Bigl(d\bigl(1-\tfrac{2\mu}{d-1}\bigr)\Bigr)^{-h}.
\end{align*}

\end{lemma}

\begin{proof}
  This has the same proof as \thref{lem:kbrw1} except that the initial value of the supermartingale
  $w_\theta(\xi_t)$ is $\theta^{-h+1}$ rather than $1$.
  We then have
  \begin{align*}
    \theta_0^{-h} \E N &\leq \theta_0^{-h+1},\\
    \theta_1\E X &\leq \theta_1^{-h+1},
  \end{align*}
  and \eqref{eq:theta_0} and \eqref{eq:theta_1} from \thref{lem:martingale} give
  the bounds on $\E N$ and $\E X$.
\end{proof}

\subsection{Estimates on tagged frogs}

Again, fix $j$, $H$, and $h$, and consider the frog model on $\TT_d^{(H+h)*}$ with the system of 
tags given previously. Recall that all frogs lose their tags at time $2^{j+1}+1$, and so
all of the following random variables count frogs only up to time $2^{j+1}$.  See Figure \ref{fig:move}. 

\begin{itemize}
  \item For $\ell\in\{A,B_0,B_1,\ldots,C_0,C_1,\ldots\}$, let
    $X_\ell$ be the number of $\ell$-tagged frogs eventually frozen at $y$.
  \item For $\ell=B_i$, $i\geq 0$, or $\ell=C_i$, $i\geq 1$,
    let $N_\ell$ be the number of frogs that received an $\ell$ tag at $\Ll_h$.
    These are the frogs that move from $\Ll_{h-1}$ to $\Ll_h$
    and change their tags to $\ell$, as well as the frogs sleeping at $\Ll_h$ woken
    by them.
  \item For $\ell=B_i$, $i\geq 0$, or $\ell=C_i$, $i\geq 1$,
    let $M_\ell$ be the number of $\ell$-tagged frogs that move from
    $\Ll_h$ back to $\Ll_{h-1}$, not counting $B_i$-tagged frogs
    that do so at times $2^j$ and on.
  \item Let $M_{C_0}$ be the number of
    $B_i$-tagged frogs for any $i$ that move from $\Ll_h$ to $\Ll_{h-1}$
    at time $2^j$ or later. These are the frogs that change tags to $C_0$.
\end{itemize}

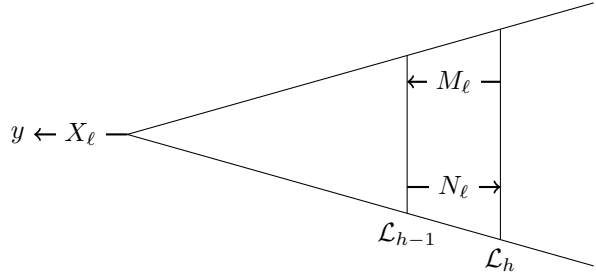
\begin{figure}
\begin{tikzpicture}[xscale = .62,yscale=.35,<->/.tip=Classical TikZ Rightarrow]
	\draw[thick,->] (0,0)-- node [midway,fill=white] {$X_\ell$} (-2,0);
	\node[above, left] at (-2,0) {$y$};
	\draw (10,5) -- (0,0) --  (10,-5); 
	\draw (6,3) -- (6,-3);
	\draw (8,4) -- (8,-4);
	\node[below] at (8,-4) {$\mathcal L_h$};
	\node[below] at (6,-3) {$\mathcal L_{h-1}$};
	\draw[thick, ->] (8,2) -- node [midway,fill=white] {$M_\ell$}  (6,2);
	\draw[thick, <-] (8,-2) -- node [midway,fill=white] {$N_\ell$}  (6,-2);
\end{tikzpicture}	
\caption{$X_\ell$ counts the number of $\ell$-tagged frogs frozen at $y$;
 $N_\ell$ counts the number of frogs that move from $\mathcal L_{h-1}$ to $\mathcal L_h$ and
 are retagged as $\ell$, plus the $\ell$-tagged frogs awoken by these frogs at $\Ll_h$;
  and $M_\ell$ counts how many $\ell$-tagged frogs move from $\Ll_h$ to $\Ll_{h-1}$. 
  This decomposes the total number of visits to $y$ as in \eqref{eq:dec}.}\label{fig:move}
\end{figure}

Recall that the total number of frogs frozen at $y$ by time $2^{j+1}$ is $X^{(j+1,H+h)}$. We have just decomposed this quantity as
\begin{align}X^{(j+1,H+h)} = X_A+\sum_{i=0}^{\infty} (X_{B_i}+X_{C_i})\label{eq:dec}
	\end{align}
Our eventual goal is bound this in expectation under the assumption that \eqref{eq:induction2}
holds for $X^{(j,H)}$, thus advancing the induction by a step.

\begin{lemma}\thlabel{lem:brw.initial.stage}
  If $\mu\leq (d-1)^2/4d$, then
  \begin{align*}
    \E N_{B_0} \leq \Bigl(1+\tfrac{2d}{d-1}\mu\Bigr)^h, \qquad\text{and}\qquad
    \E X_A \leq \Bigl(d\bigl(1-\tfrac{2\mu}{d-1}\bigr)\Bigr)^{-1}.
  \end{align*}
\end{lemma}
\begin{proof}
  This follows immediately from \thref{lem:kbrw1}.
\end{proof}

\begin{lemma}\thlabel{lem:brw.return.stage}
  For any $i\geq 0$, suppose either that $\ell=B_i$ and $\ell^+=B_{i+1}$, or
  that $\ell=C_i$ and $\ell^+=C_{i+1}$. If $\mu\leq (d-1)^2/4d$, then
  \begin{align*}
    \E N_{\ell^+} \leq\Bigl(1+\tfrac{2d}{d-1}\mu\Bigr)\E M_\ell, 
    \qquad \text{and}\qquad
    \E X_{\ell} \leq \Bigl(d\bigl(1-\tfrac{2\mu}{d-1}\bigr)\Bigr)^{-h}\E M_\ell.
  \end{align*}
\end{lemma}
\begin{proof}
  Enumerate the $\ell$-tagged frogs that return from $\Ll_h$ to $\Ll_{h-1}$
  as frogs $1,\ldots,M_\ell$. 
  Let $v_1,\ldots,v_{M_{\ell}}\in\Ll_h$ be the vertices
  that the frogs emerge from.
  For each $1\leq k\leq M_{\ell}$, we define random variables $N(k)$ and $X(k)$
  that give the portions of $N_{\ell^+}$ and $X_{\ell}$ that are attributable
  to frog~$k$, in a sense that we will explain. We then estimate $N(k)$ and $X(k)$ using
  the branching random walk bounds.
  
  To define $N(k)$ and $X(k)$, consider the following modified frog model on $\TT^{(H+h)*}_d$:
  \begin{itemize}
    \item at all vertices at levels $1$ to $h$ except for the ones on the path from the root
      to $v_k$, place the same sleeping frogs as in the current realization of the original frog model
      on $\TT^{(H+h)*}_d$;
    \item place an initially active frog at $v_k$ that follows the path of frog~$k$
      starting from when it moves from $v_k$ back to $\Ll_{h-1}$;
    \item freeze all frogs on on arrival at $\Ll_h$ and at $y$.
  \end{itemize}
  We define $N(k)$ and $X(k)$ as the number of frogs eventually frozen at $\Ll_h$
  and $y$, respectively, in this frog model. As usual, we include
  the frogs woken at $\Ll_h$ and immediately frozen in the count of $N(k)$.
  We claim that $N_{\ell^+}\leq \sum_{k=1}^{M_\ell}N(k)$. This is because any $\ell$-tagged
  frog counted by $N_{\ell^+}$ either is one of frogs $1,\ldots,k$ or is spawned by a sequence of
  frogs at levels $1,\ldots,n$ originating with one of frogs $1,\ldots,k$. Hence,
  any frog counted by $N_{\ell^+}$ must also be counted by $X(k)$ for at least one
  $k\in\{1,\ldots,M_\ell\}$. By the same argument, $X_\ell\leq\sum_{k=1}^{M_\ell}X(k)$.
  
  The conditional distributions given $M_\ell$ of
  $N(k)$ and $X(k)$, respectively, are exactly those of
  $N$ and $X$ from \thref{lem:kbrw2}. Applying this lemma,
  \begin{align*}
    \E [N_{\ell^+}\mid M_\ell] &\leq \Bigl(1+\tfrac{2d}{d-1}\mu\Bigr)M_\ell,\\
    \E [X_\ell\mid M_\ell] &\leq \Bigl(d\bigl(1-\tfrac{2\mu}{d-1}\bigr)\Bigr)^{-h}M_\ell.
  \end{align*}
  Now take expectations to complete the proof.
\end{proof}

\begin{lemma}\thlabel{lem:inductive.stage}
  Suppose that \eqref{eq:induction2} holds for the fixed $j$ and $H$ used in the definitions
  of $X_\ell$, $N_\ell$, and $M_\ell$. If $\mu\leq (d-1)^2/4d$, then
  for $\ell=B_i$, $i\geq 0$, or $\ell=C_i$, $i\geq 1$,
  \begin{align*}
    \E M_\ell &\leq \frac{.8\E N_\ell}{1 + \frac{2d}{d-1}\mu}.
  \end{align*}
\end{lemma}
\begin{proof}
  We claim that
  \begin{align}\label{eq:cond.bound}
    \E [M_\ell\mid N_\ell] &\leq \frac{.8 N_\ell}{1 + \frac{2d}{d-1}\mu},
  \end{align}
  from which the lemma follows by taking expectations. 
  Roughly speaking, we want to show that for each frog acquiring an $\ell$ tag
  at $v\in \Ll_h$, the expected number of
  $\ell$-tagged frogs moving from $v$ back to $\Ll_{h-1}$ is at most 
  $.8/\bigl(1+2d\mu/(d-1)\bigr)$.
  This follows from \eqref{eq:induction2}, as we will now show
  in detail.

  Enumerate the frogs counted by $N_\ell$ as frogs $1,\ldots,N_\ell$. Recall that these include
  both the frogs that move from $\Ll_{h-1}$ to $\Ll_h$ and receive an $\ell$ tag,
  as well as the frogs woken at $\Ll_h$ by them. For $1\leq k\leq N_{\ell}$, let
  $v_k\in\Ll_h$ be the vertex where frog~$k$ received its $\ell$ tag.
  Note that the same vertices will appear multiple times in $v_1,\ldots,v_{N_\ell}$, though all frogs in the
  list are unique.
  For each $k$, we will define a random variable $M(k)$ that gives the number of
  frogs counted by $M_\ell$ attributable to frog~$k$. As we did in the previous
  lemma, we then bound $M(k)$, this time using \eqref{eq:induction2}.
  
  To define $M(k)$, let $y_k$ be the parent of $v_k$, and consider the following frog model
  on $\{y_k\}\cup \TT^{H+h}_d(v_k)$:
  \begin{itemize}
    \item at all descendants of $v_k$, place the same sleeping frogs as in the current realization
      of the original frog model on $\TT_d^{(H+h)*}$;
    \item place an initially active frog at $v_k$ that follows that path of frog~$k$ starting
      from its arrival at $v_k$;
    \item freeze all frogs on visiting $y_k$.
  \end{itemize}
  We then define $M(k)$ as the number of frogs frozen at $y_k$ in the first $2^j$ steps of this
  frog model.
  We claim that $M_{\ell}\leq \sum_{i=1}^{N_{\ell}}M(k)$.
  To justify this, we first observe that any return from $v_k$ to $y_k$
  counted by $M_\ell$ must occur within $2^j$ steps
  of when $v_k$ is first visited by a frog that changes its label to $\ell$.
  When $\ell=B_i$, this is because
  $M_\ell$ only counts returns up to time~$2^j$. When $\ell=C_i$, it is because the first
  visit to $v_k$ by a frog receiving an $\ell$ tag occurs after time~$2^j$, and $M_\ell$ only counts
  returns up to time $2^{j+1}$.
  Now, any $\ell$-tagged frog counted by $M_\ell$ is either one of frogs $1,\ldots,N_\ell$ or is
  spawned by a sequence of frogs at level $h+1$ and beyond in $\TT^{(H+h)*}_d$ originating
  with one of these frogs. It is thus counted by $M(k)$ for some $1\leq k\leq N_\ell$.
  
  Observe that the frog model defining $M(k)$ is just a disguised version
  of the frog model on $\TT^{H*}_d$ considered in \thref{thm:induction}.
  Hence, the distribution of $M(k)$ conditional on $N_\ell$
  is exactly that of $X^{(j,H)}$. Applying \eqref{eq:induction2}, we have 
  \begin{align*}
    \E[M(k)\mid N_\ell]\leq \frac{.8}{1 + \frac{2d}{d-1}\mu}.
  \end{align*}
  Summing this over all $k$ to bound $\E[M_{\ell}\mid N_\ell]$ proves \eqref{eq:cond.bound}.
 \end{proof}

\thref{lem:brw.return.stage} gives bounds on $\E N_{B_{i+1}}$ and $\E N_{C_{i+1}}$
in terms of $\E M_{B_i}$ and $\E M_{C_i}$, and \thref{lem:inductive.stage} gives bounds on $\E M_{B_i}$
and $\E M_{C_i}$ in terms of $\E N_{B_i}$ and $\E N_{C_i}$. Together, these bounds
show that $\E N_{B_i}$, $\E N_{C_i}$, $\E M_{B_i}$, and $\E N_{C_i}$ decay exponentially in $i$.
\begin{lemma}\thlabel{lem:decay}
  For any $i\geq 0$, suppose either that $\ell=B_i$ and $\ell^+=B_{i+1}$, or
  that $\ell=C_i$ and $\ell^+=C_{i+1}$. If $\mu\leq (d-1)^2/4d$, then
  \begin{align*}
    \E M_{\ell^+} &\leq .8 \E M_{\ell},\\
    \E N_{\ell^+} &\leq .8 \E N_{\ell}.
  \end{align*}
  Consequently,
  \begin{align*}
    \sum_{i=0}^{\infty} \E M_{B_i} &\leq 5 \E M_{B_0},&
    \sum_{i=0}^{\infty} \E M_{C_i} &\leq 5 \E M_{C_0},\\
    \sum_{i=0}^{\infty} \E N_{B_i} &\leq 5 \E N_{B_0},\text{ and}&
    \sum_{i=0}^{\infty} \E N_{C_i} &\leq 5 \E N_{C_0}.
  \end{align*}
\end{lemma}
\begin{proof}
  The bounds on $\E M_{\ell^+}$ and $\E N_{\ell^+}$ follow immediately
  from \thref{lem:brw.return.stage,lem:inductive.stage}. The other bounds
  are consequences of summing geometric series.
\end{proof}

 \begin{lemma}\thlabel{lem:all.awake.stage}
  \begin{align*}
    \E M_{C_0}\leq c_1\mu 2^{j+1} \sum_{i=0}^{\infty} \E N_{B_i}.
  \end{align*}
\end{lemma}
\begin{proof}
  This proof is just as for \thref{lem:brw.return.stage,lem:inductive.stage}, except we
  use the all-awake bound in place of the branching random walk bounds or the inductive
  hypothesis. In more detail,
  fix a nonnegative integer $i$, and number the frogs that received a $B_i$ tag at $\Ll_h$
  as $1,\ldots,N_{B_i}$. These are made up of the $B_{i-1}$-tagged frogs that
  moved from $\Ll_{h-1}$ to $\Ll_h$ (where $B_{-1}=A$), as well as the frogs at $\Ll_h$
  that these frogs woke up. Let $v_1,\ldots,v_{N_{B_i}}\in\Ll_h$ be the sites where
  these frogs get their $B_i$ tags.
  In a similar argument as we used in \thref{lem:brw.return.stage,lem:inductive.stage},
  we define a random variable $M(k)$ giving the number of frogs counted by $M_{C_0}$
  that are attributable to frog~$k$. 
  Let $y_k$ be the parent of $v_k$, and define a frog model
  on $\{y_k\}\cup \TT^{H+h}_d(v_k)$ exactly as in \thref{lem:inductive.stage}.
  Define $M(k)$ as the number of frogs frozen at $y_k$ in the first $2^{j+1}$
  steps of this frog model.
  Let $M^i_{C_0}$ be the number of $B_i$-tagged frogs in the original model that move from $\Ll_h$
  to $\Ll_{h-1}$ between times $2^j+1$ and $2^{j+1}$ and change tags to $C_0$.  
  By similar reasoning as in the previous lemmas,
  we have $M^i_{C_0}\leq \sum_{k=1}^{N_{B_i}}M(k)$.
  
  Conditional on $N_{B_i}$, the distribution of $M(k)$ is stochastically dominated
  by the random variable
  $W$ from \thref{lem:all_awake}. Applying the bound from \thref{lem:all_awake} and
  summing over all $k$, we get
  \begin{align*}
    \E\bigl[ M^i_{C_0} \mid N_{B_i} \bigr] &\leq c_1 N_{B_i} \mu 2^{j+1}.
  \end{align*}
  Therefore,
  \begin{align*}
    \E M_{C_0} = \sum_{i=0}^{\infty} \E M^i_{C_0} &\leq c_1\mu 2^{j+1}\sum_{i=0}^{\infty} \E N_{B_i}.\qedhere
  \end{align*}
\end{proof}

We are now ready to advance the induction in \thref{thm:induction}.
\begin{prop}\thlabel{prop:induction.step}
  There exists an absolute constant $C$ so that the following statement holds.
  Suppose that for some specific choice of $j$, $H$, $d$, and $\mu$ with
  $j,H\geq 1$, $d\geq 2$, and $\mu\leq d/100$, the inductive hypothesis
  \eqref{eq:induction2} holds.
  Then, for any
  \begin{align}\label{eq:s}
    h \geq \frac{C\bigl(j + \log(1+\mu)\bigr)}{\log\bigl( \frac{d}{1+4\mu}\bigr)},
  \end{align}
  we have
  \begin{align*}
    \E X^{(j+1,H+h)} \leq \frac{.8}{1 + \frac{2d}{d-1}\mu}.
  \end{align*}
\end{prop}
\begin{proof}
Using the decomposition of $X^{(j+1, H+h)}$ given in \eqref{eq:dec}, our goal is to show that
  \begin{align*}
    \E\Biggl[ X_A + \sum_{i=0}^{\infty} X_{B_i} + \sum_{i=0}^{\infty} X_{C_i} \Biggr] 
      &\leq\frac{.8}{1 + \frac{2d}{d-1}\mu}.
  \end{align*}
  In successive lines, we apply 
  \thref{lem:brw.return.stage,lem:decay,lem:inductive.stage,lem:brw.initial.stage} to obtain
  \begin{align*}
    \sum_{i=0}^{\infty}\E X_{B_i} &\leq 
      \Bigl(d\bigl(1-\tfrac{2\mu}{d-1}\bigr)\Bigr)^{-h}\sum_{i=0}^{\infty}\E M_{B_i}\\
      &\leq 5\Bigl(d\bigl(1-\tfrac{2\mu}{d-1}\bigr)\Bigr)^{-h}\E M_{B_0}\\
      &\leq 5\Bigl(d\bigl(1-\tfrac{2\mu}{d-1}\bigr)\Bigr)^{-h}
        \frac{.8\E N_{B_0}}{1 + \frac{2d}{d-1}\mu}\\
      &\leq 5\Biggl(\frac{1+\tfrac{2d}{d-1}\mu}{d\bigl(1-\tfrac{2\mu}{d-1}\bigr)}\Biggr)^h
                \Biggl(\frac{.8}{1 + \frac{2d}{d-1}\mu}\Biggr)
        \leq \Biggl(\frac{1+\tfrac{2d}{d-1}\mu}{d\bigl(1-\tfrac{2\mu}{d-1}\bigr)}\Biggr)^h
                \Biggl(\frac{4}{1 + 2\mu}\Biggr).
  \end{align*}
  Applying \thref{lem:brw.return.stage,lem:decay,lem:all.awake.stage},
  \begin{align*}
    \sum_{i=0}^{\infty}\E X_{C_i} &\leq 
      \Bigl(d\bigl(1-\tfrac{2\mu}{d-1}\bigr)\Bigr)^{-h}\sum_{i=0}^{\infty}\E M_{C_i}\\
      &\leq 5\Bigl(d\bigl(1-\tfrac{2\mu}{d-1}\bigr)\Bigr)^{-h}\E M_{C_0}\\
      &\leq 5\Bigl(d\bigl(1-\tfrac{2\mu}{d-1}\bigr)\Bigr)^{-h}c_1\mu 2^{j+1}\sum_{i=0}^{\infty}\E N_{B_i},
  \end{align*}
  and then applying \thref{lem:decay,lem:brw.initial.stage},
  \begin{align*}
    \sum_{i=0}^{\infty}\E X_{C_i}  
      &\leq 25\Bigl(d\bigl(1-\tfrac{2\mu}{d-1}\bigr)\Bigr)^{-h}c_1\mu 2^{j+1}\E N_{B_0}\\
      &\leq 25\Biggl(\frac{1+\tfrac{2d}{d-1}\mu}{d\bigl(1-\tfrac{2\mu}{d-1}\bigr)}\Biggr)^{h}c_1\mu 2^{j+1}.
  \end{align*}
  Applying these bounds together with the estimate on $\E X_A$ from \thref{lem:brw.initial.stage},
  \begin{align*}
    \E\Biggl[ X_A + \sum_{i=0}^{\infty} X_{B_i} + \sum_{i=0}^{\infty} X_{C_i} \Biggr]
      &\leq  \frac{1}{d\bigl(1-\tfrac{2\mu}{d-1}\bigr)} + 
        \Biggl(\frac{1+\tfrac{2d}{d-1}\mu}{d\bigl(1-\tfrac{2\mu}{d-1}\bigr)}\Biggr)^{h}
          \biggl(\frac{4}{1+2\mu} + 25c_1\mu 2^{j+1}\biggr)\\
      &\leq \frac{1}{.96d} + 
        \biggl(\frac{1+4\mu}{.96d}\biggr)^{h}
          \Bigl(4 + 50c_1\mu 2^{j}\Bigr).
  \end{align*}
  From \eqref{eq:s},
  \begin{align*}
    \biggl(\frac{1+4\mu}{.96d}\biggr)^{h}
      &\leq \exp\Biggl[  \biggl(-\log\Bigl(\frac{d}{1+4\mu}\Bigr) + \log\Bigl(\frac{25}{24}\Bigr)\biggr)
         \frac{C\bigl(j + \log(1+\mu)\bigr)}{\log \bigl(\frac{d}{1+4\mu}\bigr)} \Biggr]\\
      &\leq \exp\biggl[ -.9C\bigl(j + \log(1+\mu)\bigr)\biggr]=e^{-.9Cj}(1+\mu)^{-.9C}.
  \end{align*}
  A bit of asymptotic analysis now shows that we can choose $C$ large enough 
  that for all $\mu\geq 0$ and $j\geq 1$,
  \begin{align*}
    \E\Biggl[ X_A + \sum_{i=0}^{\infty} X_{B_i} + \sum_{i=0}^{\infty} X_{C_i} \Biggr]
      &\leq \frac{1}{.96d} + \frac{.2}{1+4\mu},
  \end{align*}
  and
  \begin{align*}
    \frac{1}{.96d} + \frac{.2}{1+4\mu}
      &=\frac{ \frac{1+4\mu}{.96d} + .2}{1+4\mu}\leq \frac{ \frac{1+.04d}{.96d} + .2}{1+4\mu}
      \leq \frac{.8}{1+4\mu}\leq \frac{.8}{1 + \frac{2d}{d-1}\mu}.\qedhere
  \end{align*}
\end{proof}

\begin{proof}[Proof of \thref{thm:induction}]
  We start by establishing \eqref{eq:induction2} when $j=1$ and $n\geq 1$.
  Consider the frog model on $\TT_d^{1*}$. Fix any $n\geq 1$.
  If the initial
  frog moves immediately to $y$, then $X^{(1,n)}=1$. If instead it moves to a child of the
  root, then no frogs can make it to $y$ by time~$2$, and $X^{(1,n)}=0$. Hence,
  $\E X^{(1,n)} = 1/(d+1)$, and this is easily seen to be less than the right-hand
  side of \eqref{eq:induction2} using our assumption that $\mu\leq d/100$.
  
  Applying \thref{prop:induction.step} inductively, \eqref{eq:induction2} holds 
  for $j\geq 2$ so long as we can show that
  \begin{align*}
    H_j\geq 1 + \sum_{i=2}^j\ceil[\Bigg]{
                \frac{C\bigl(i + \log(1+\mu)\bigr)}{\log\bigl(\frac{d}{1+4\mu}\bigr)}},
  \end{align*}
  where $C$ is the constant from \thref{prop:induction.step}.
  Indeed, it is straightforward to compute that
  \begin{align*}
    1 + \sum_{i=2}^j\ceil[\Bigg]{
                \frac{C\bigl(i + \log(1+\mu)\bigr)}{\log\bigl(\frac{d}{1+4\mu}\bigr)}}
       &= O\biggl(\frac{j^2 + j\log(1+\mu)}{\log\bigl(\frac{d}{1+\mu}\bigr)}\biggr).\qedhere
  \end{align*}
\end{proof}

\appendix

\section{Stochastic comparison results for the frog model}
\label{sec:comparison}
In this section, we outline the results of  \cite{JJ3_order}, which allow us to compare two
frog models on the same graph with different initial conditions. If the distribution of
frog counts in the first model stochastically dominates the distribution in the other, then
certain statistics of the first model will dominate the corresponding statistics in the other.
This is a trivial fact with the typical definition of stochastic domination. The
strength of these results is that they apply to less conventional stochastic orders, one of which is
the \emph{probability generating function order},
whose name we abbreviate to \emph{pgf order}.

For two probability measures $\pi_1$ and $\pi_2$ on the nonnegative real numbers, 
we say that $\pi_1$ is smaller than $\pi_2$ in the 
pgf order, denoted $\pi_1\pgfprec \pi_2$,
if for $X\sim\pi_1$ and $Y\sim\pi_2$ and all $t\in(0,1)$, it holds that
$\E t^X\geq\E t^Y$. 
We also write $X\pgfprec Y$ to mean that the law of $X$ is stochastically
smaller in the pgf order than the law of $Y$, and we also use mixed expressions like $X\pgfprec\pi_2$
in the obvious way.
See the introduction of \cite{JJ3_order} for more on 
the pgf order and its relations to other stochastic orders.

We now present the result from \cite{JJ3_order} as we will apply it in this paper.

\begin{lemma}\thlabel{cor:comparison}
  Consider two frog models on $\TT^n_d$ with initial frog counts given by $(\eta(v))_{v}$ 
  and $(\eta'(v))_v$ for
  $v\in\TT_d^n\setminus\{\emptyset\}$. Assume that both counts are independent.
  Suppose that $\eta(v)\pgfprec\eta'(v)$ for all $v$.
  Let $N$ and $N'$ be the number of leaves visited in the two models by some given time,
  and let $R$ and $R'$ be the number of visits to the root in the two models by some given time.
  Then $N\pgfprec N'$ and $R\pgfprec R'$.
\end{lemma}
\begin{proof}
  By \cite[Theorem~3]{JJ3_order}, this holds once we prove that the number of leaves visited
  by time~$t$ and the number of visits to the root by time~$t$ are \emph{continuous
  pgf statistics}. For the first statistic, this is a very slight variation of 
  \cite[Proposition~21]{JJ3_order} and has a nearly identical proof.
  For the second statistic, it is a consequence of \cite[Proposition~4]{JJ3_order}, if we think
  of the frog models as having frog paths stopped at time~$t$.
\end{proof}

We mention two basic facts about the pgf order.
First, if $X\pgfprec Y$, then $\E X\leq \E Y$. Second, for any distribution $\pi$ with expectation
$\mu$ or less, we have $\pi\pgfprec\Ber(\mu)$, recalling
our definition of $\Ber(\mu)$ for $\mu>1$ as the unique distribution on 
$\{ \floor{\mu} , \ceil{\mu}\}$ with mean $\mu$. This fact is proven in \cite[Proposition~15(b)]{JJ3_order}
for a different stochastic relation known as the \emph{increasing concave order}, and it follows
that it holds for the pgf order since domination in the increasing concave order implies
domination in the pgf order (see \cite[Section~2]{JJ3_order}).

\section{Miscellaneous concentration inequalities}\label{sec:concentration}

The following two bounds appear verbatim in \cite[Appendix~C]{HJJ_shape}.
Both are standard results that follow from bounding the moment generating function and applying
Markov's inequality.

\begin{prop}\thlabel{prop:Poi.bound}
  Let $\E Y=\lambda$, and suppose either that $Y$ is Poisson or that $Y$ is a sum of independent
  random variables supported on $[0,1]$.
  For any $0<\alpha<1$,
  \begin{align*}
    \P[Y\leq\alpha\lambda] &\leq \exp\biggl(-\frac{(1-\alpha)^2\lambda}{2}\biggr),
  \end{align*}
  and for any $\alpha>1$,
  \begin{align*}
    \P[Y\geq\alpha\lambda] &\leq \exp\Biggl(-\frac{(\alpha-1)\lambda}{\frac23 + \frac{2}{\alpha-1}}\Biggr).
  \end{align*}
\end{prop}

\begin{prop}\thlabel{prop:exp.sum.bound}
  Let $(X_i)_{i=1}^n$ be a collection of independent random variables satisfying
  \begin{align*}
    \P[X_i \geq \ell] \leq Ce^{-b\ell}
  \end{align*}
  for some $C$ and $b>0$ and all $\ell\geq 1$. 
  Then for any $b'>0$, there exists $C'$ depending on $C$, $b$, and $b'$ such that
  \begin{align*}
    \P\Biggl[\sum_{i=1}^n X_i \geq C'n \Biggr] &\leq e^{-b'n}.
  \end{align*}
  We can take $C'=2(b'+C)/b$.
\end{prop}

Next, we give a more refined version of the previous proposition that applies to random variables
that are exactly geometrically distributed.
\begin{prop}\thlabel{prop:geo.bound}
  Let $(G_i)_{i\geq 1}$ be a collection of independent random variables with $G_i$ geometrically
  distributed on $\{1,2,\ldots\}$ with parameter $p$. Let $\mu:=\E G_i=1/p$. For any $\lambda\geq 2$,
  \begin{align*}
    \P\bigl[G_1+\cdots + G_n \geq \lambda n\mu \bigr] 
       &\leq  \exp\biggl[-n\biggl(\frac{\lambda}{2}-1\biggr)\biggr].
  \end{align*}
\end{prop}
\begin{proof}
  If $G_1+\cdots+G_n\geq k$, then in the first of $k$ independent trials with success probability
  $p$, there were at most $n$ successes. Thus, by \thref{prop:Poi.bound},
  \begin{align*}
    \P[G_1+\cdots+G_n\geq k] &= \P\bigl[ \Bin(k,p)\leq n \bigr]
     \leq \exp\biggl(-\frac{(1-n/kp)^2kp}{2}\biggr).
  \end{align*}
  Substituting $k=\lambda n\mu$ and $p=1/\mu$ gives
  \begin{align*}
    \P[G_1+\cdots+G_n\geq k] &\leq \exp\biggl(-\frac{(\lambda-2+\lambda^{-1}) n}{2}\biggr)
      \leq \exp\biggl(-\frac{(\lambda-2) n}{2}\biggr).\qedhere
  \end{align*}
\end{proof}

Last, we extend the exponential concentration bound for a single Poisson or binomial random variable
to an entire sequence, via a union bound:
\begin{lemma}\thlabel{lem:Poi.sequence}
  Let $\gamma_1\geq 2\gamma_2>0$, and let $\gamma_1\geq 8$.
  Suppose that $X_i$ is Poisson or binomial with mean $\gamma_1 i$ for all $i\geq k$, with no
  assumption on the joint distribution of $(X_i)_{i\geq k}$. Then
  \begin{align*}
    \P[ \text{$X_i<\gamma_2i$ for some $i\geq k$} ] 
      &\leq 2\exp\biggl( -\frac{(1- \gamma_2/\gamma_1)^2\gamma_1 k}{2}\biggr)
  \end{align*}
\end{lemma}
\begin{proof}
  By \thref{prop:Poi.bound},
  \begin{align*}
    \P[X_i < \gamma_2 i] &\leq \exp\biggl( -\frac{(1- \gamma_2/\gamma_1)^2\gamma_1 i}{2}\biggr).
  \end{align*}
  Applying a union bound over all $i\geq k$ and summing the geometric series gives
  \begin{align*}
    \P[ \text{$X_i<\gamma_2i$ for some $i\geq k$} ] &\leq 
      \frac{\exp\Bigl( -\frac{(1- \gamma_2/\gamma_1)^2\gamma_1 k}{2}\Bigr)}
           {1 - \exp\Bigl( -\frac{(1- \gamma_2/\gamma_1)^2\gamma_1}{2}\Bigr)},
  \end{align*}
  and
  \begin{align*}
    1 - \exp\biggl( -\frac{(1- \gamma_2/\gamma_1)^2\gamma_1}{2}\biggr)&\geq 
      1 - \exp\biggl( -\frac{(1- 1/2)^2(8)}{2}\biggr)
      \geq \frac12.\qedhere
  \end{align*}
\end{proof}

\bibliographystyle{amsalpha}
\bibliography{frog_paper_cover.bib}

\end{document}